\documentclass[12pt]{article}

\usepackage[margin=1.17in]{geometry}

\usepackage{amssymb,amsmath}
\usepackage{amsthm,enumitem}
\usepackage{hyperref}
\usepackage{mathrsfs}
\usepackage{textcomp}
\hypersetup{
colorlinks,%
citecolor=black,%
filecolor=black,%
linkcolor=black,%
urlcolor=black
}
\usepackage[abbrev,nobysame]{amsrefs}
\usepackage{todonotes}

\usepackage{verbatim}

\usepackage{sectsty}

\makeatletter

\newdimen\bibspace
\makeatother

\makeatletter

\newtheorem{Theorem}{Theorem}[section]
\newtheorem*{Theorem*}{Theorem A}
\newtheorem{Lemma}[Theorem]{Lemma}

\newtheorem*{Assumption*}{Assumption (H)}
\newtheorem*{Assumption**}{Assumption (G)}

\def\XXint#1#2#3{{\setbox0=\hbox{$#1{#2#3}{\int}$}
\vcenter{\hbox{$#2#3$}}\kern-.5\wd0}}

\newcommand{\Om}{\Omega}                
           \newcommand{\ud}{\mathrm{d}}
\newcommand{\be}{\begin{equation}}      \newcommand{\ee}{\end{equation}}

\newcommand{\A}{\mathcal{A}}
\newcommand{\LL}{\mathcal{L}}
\newcommand{\C}{\mathcal{C}}

\newcommand{\R}{\mathbb{R}}              
\newcommand{\D}{\mathcal{D}}

\newcommand{\M}{\mathscr{M}}

\newcommand{\abs}[1]{\lvert#1\rvert}

\begin{document}

\title{\textbf{Sharp global Alexandrov estimates and entire solutions of Monge-Amp\`ere equations}\bigskip}

\author{\medskip  Tianling Jin,\footnote{T. Jin was partially supported by NSFC grant 12122120, and Hong Kong RGC grants GRF 16304125, GRF 16303624 and GRF 16303822.}\quad Xushan Tu, \quad
Jingang Xiong\footnote{J. Xiong was partially supported by NSFC grants 12325104.}}

\date{}

\maketitle

\begin{center}
\vspace{-1cm}
\emph{{\small Dedicated to YanYan Li on the occasion of his 65th birthday, with \\ admiration and friendship.\medskip}}
\end{center}

\begin{abstract}  
This paper continues our work \cite{jin2025abp} on sharp Alexandrov estimates. We obtain a sharp global uniform distance estimate from a convex function to the class of unimodular convex quadratic polynomials in terms of the total variation of its Monge-Amp\`ere defect measure relative to Lebesgue measure. The estimate has an explicit optimal constant, and the inequality is strict in the regime of positive finite defect mass. In this regime we further prove asymptotic rigidity at infinity: every such convex function admits a unique quadratic asymptote with an explicit convergence rate, and satisfies a sharp affine invariant global Alexandrov estimate with equality if and only if the function solves the isolated singularity problem or the hyperplane obstacle problem. Standard subsolution methods are not well suited to this measure-theoretic setting and typically do not yield sharp constants, while the sharp Alexandrov estimates developed in our earlier work \cite{jin2025abp} play a central role here.

As an application, for entire solutions of Monge-Amp\`ere equations with multiple (possibly infinitely many) isolated singularities, we give an explicit quantitative mass-separation condition ensuring strict convexity and hence smoothness away from the set of the isolated singularities. 
\medskip

\noindent{\it Keywords}: Alexandrov estimates, Monge-Amp\`ere equation, entire solution, obstacle problem, isolated singularity.

\medskip

\noindent {\it MSC (2010)}: Primary 35B25; Secondary 35J96, 35R35.

\end{abstract}

\section{Introduction} 

In this paper, we begin by considering the following question for entire convex functions  $u: \R^n\to \R$. If $u$ stays finitely far away from some convex quadratic polynomial at infinity, how can one bound,  in the $L^\infty(\R^n)$ norm, the distance from $u$ to the class of convex quadratic polynomials?

We study this question in terms of the Monge-Amp\`ere measure associated with $u$. 
For a convex function $u:\R^n\to\R$, we denote by $\M u$ its associated Monge-Amp\`ere measure. We write $\LL$ as the Lebesgue measure on $\R^n$. For a signed Borel measure $\mu$ on $\R^n$, $|\mu|$ denotes its total variation measure, and $|\mu|(\R^n)$ its total mass.

Let $\A_n$ denote the set of positive definite symmetric $n \times n$ matrices with determinant $1$, and let $\omega_n$ be the Lebesgue measure of the unit ball in $\R^n$. Define
\begin{equation}\label{eq:def dn00}
d_{n,0}: =   \frac{\Gamma\left(\frac{1}{n}\right)\Gamma\left(\frac{n-2}{n}\right)}{2n\Gamma\left(\frac{n-1}{n}\right)}.
\end{equation}

Our first result is a sharp global deviation bound.
\begin{Theorem}\label{thm:global deviation} 
Suppose $n \ge 3$.  
Then for every entire convex function $u:\R^n\to\R$, we have
\begin{equation}\label{eq:global deviation}
\inf_{A \in \A_n, b \in \mathbb{R}^n, c \in \mathbb{R}}\left\| u - \left(\frac{1}{2}x^{\top}Ax  + b \cdot x+c\right) \right\|_{L^{\infty}(\R^n)} \leq 2^{-\frac{2}{n}} d_{n,0}  \left(\omega_n^{-1}|\M u-\LL|(\R^n)\right)^\frac{2}{n}.
\end{equation}  
Moreover, the inequality is strict whenever $0<|\M u-\LL|(\R^n)<\infty$, and the constant $2^{-\frac{2}{n}} d_{n,0}$ is sharp for each fixed value of $|\M u-\LL|(\R^n)$ in $(0,\infty)$.
\end{Theorem}

When $|\M u-\LL|(\R^n)=0$, the problem reduces, using Caffarelli's regularity theory \cite{caffarelli1990ilocalization,caffarelli1990interiorw2p,caffarelli1991regularity,caffarelli1995topics},  to the classical Jörgens-Calabi-Pogorelov theorem \cite{calabi1958improper, jorgens1954losungen, pogorelov1972improper} : every convex solution of
\begin{equation*}\label{eq:Monge-Ampère equation entire}
\det D^2 u = 1 \quad \text{in } \R^n
\end{equation*}
must be a quadratic polynomial. When $|\M u-\LL|(\R^n)=\infty$, the estimate \eqref{eq:global deviation} is immediate since its right-hand side is infinite. Hence the substantive content of Theorem \ref{thm:global deviation} concerns the intermediate, i.e.,  positive finite defect mass, regime
\[
0<|\M u-\LL|(\R^n)<\infty,
\]
where we obtain a genuinely quantitative and sharp deviation bound.

In this regime, we also prove asymptotic rigidity at infinity together with a sharp global Alexandrov estimate.

\begin{Theorem}\label{thm:abp global rigidity} 
Suppose $n \geq 3$. Let $u:\R^n\to\R$ be an entire convex function such that $0<|\M u-\LL|(\R^n)<\infty$.
Then:
\begin{itemize}
\item[(i) .] (\emph{Rigidity at infinity}) There exist $A \in \A_n$, $b \in \mathbb{R}^n$, and $c \in \mathbb{R}$ such that
\begin{equation}\label{eq:asymptotic-quadratic}
\limsup_{|x| \to \infty} \left| u(x) - \left(\frac{1}{2}x^{\top}Ax  + b \cdot x + c\right) \right| = 0.
\end{equation} 
\item[(ii).] (\emph{Global Alexandrov estimate}) Moreover, with the same $A,b,c$, we have
\begin{equation}\label{eq:abp-global}
\left\| u(x) - \left(\frac{1}{2}x^{\top}Ax  + b \cdot x + c\right) \right\|_{L^{\infty}(\R^n)} \leq d_{n,0} \left(\omega_n^{-1}|\M u-\LL|(\R^n)\right)^\frac{2}{n},
\end{equation}
where $d_{n,0}$ is defined in \eqref{eq:def dn00}. 
\item[(iii).] (\emph{Sharpness}) Equality in \eqref{eq:abp-global} holds if and only if, up to a unimodular affine transformation\footnote{An affine transformation $T : \mathbb{R}^n \to \mathbb{R}^n$ is of the form $T x = A x + b$ for some $n \times n$ matrix $A$ and vector $b \in \mathbb{R}^n$. It is called unimodular if $|\det A| = 1$.} and addition of linear functions, 
\[
u(x)=\int_0^{|x|} \left( r^n + a^n \right)^{\frac{1}{n}} \ud r \quad\mbox{or}\quad u(x)=\int_{0}^{|x|} \max\left\{ r^n - a^n,\, 0 \right\}^{\frac{1}{n}}   \ud r,
\]
where
\begin{equation}\label{eq:total-variation}
a:= (\omega_n ^{-1}|\M u-\LL|(\mathbb{R}^n))^{\frac{1}{n}}.
\end{equation}
\end{itemize}
\end{Theorem}

The estimate \eqref{eq:abp-global} is affine invariant, and it provides a sharp global version of our recent optimal Alexandrov estimates \cite{jin2025abp} in bounded domains (see Theorem \ref{thm:ordera2}). 

Note also that the exponent $\frac{2}{n}$ in the right-hand side of \eqref{eq:abp-global} is different from the exponent $\frac{1}{n}$ in the classical Alexandrove estimate (see \eqref{eq:abp gene 1/n}). 

We also have a pointwise estimate, which in particular quantifies the convergence rate to the quadratic asymptote in \eqref{eq:asymptotic-quadratic}. 

\begin{Theorem}\label{thm:global decay}
Suppose $n \geq 3$. Let $u:\R^n\to\R$ be an entire convex function such that $0<|\M u-\LL|(\R^n)<\infty$. Suppose $u$ satisfies \eqref{eq:asymptotic-quadratic}.
Then, there exists a positive constant $C(n)$ depending only on $n$ such that for every $\rho>0$, we have the quantitative estimate
\begin{equation}\label{eq:global decay}
\begin{split}
&\left| u(x) - \left(\frac{1}{2} x^{\top} A x + b \cdot x + c\right) \right|\\
& \le  d_{n,0}\bigl(\omega_n^{-1}|\M u-\LL|(E_A(x,\rho))\bigr)^{\frac{2}{n}} + C(n) \frac{|\M u-\LL|(\R^n)}{\rho^{n-2}} ,
\end{split}
\end{equation}
where  
$E_A(x, \rho) := \{ y \in \mathbb{R}^n : (y-x)^{\top} A (y-x) < \rho^2 \}.$

In particular, by choosing, e.g., $\rho = \tfrac{1}{2}\sqrt{x^{\top} A x}$, we obtain a pointwise decay estimate.
\end{Theorem}

Define the signed measure
\[
\mu:=\M u-\LL.
\]
In this notation, $u$ solves the Monge-Amp\`ere equation
\begin{equation}\label{eq:global-ma-eq-mu}
\det D^2 u = 1 + \mu \quad \text{in } \mathbb{R}^n
\end{equation}
in the Alexandrov sense. In particular, the convex functions in Theorems \ref{thm:global deviation}, \ref{thm:abp global rigidity} and \ref{thm:global decay} with $0<|\M u-\LL|(\R^n)<\infty$ may be viewed as convex solutions of \eqref{eq:global-ma-eq-mu} for which $1 + \mu$ is a non-negative Borel measure and satisfies $0<|\mu|(\mathbb{R}^n)<\infty$.

Caffarelli-Li \cite{caffarelli2003extension} studied rigidity for \eqref{eq:global-ma-eq-mu} in the absolutely continuous setting $\mu=(f(x)-1)\,\ud x$, assuming $f \in C^0(\mathbb{R}^n)$, $(f-1)$ is compactly supported, and $0<\inf_{\R^n}f <\sup_{\R^n}f < \infty$. They proved that
\begin{itemize}
\item[(i)] For $n = 2$, 
\begin{equation*}\label{eq:asympotic at infty n2}
u(x) = \frac{1}{2}x^{\top}Ax  + b \cdot x + \left(\frac{1}{2\pi}\int_{\mathbb{R}^2} f(x) \ud x\right) \log \sqrt{x^{\top}Ax} + c + O(|x|^{-1}) \quad \text{as } x \to \infty,
\end{equation*}
\item[(ii)] For $n \geq 3$,
\begin{equation}\label{eq:asympotic at infty n3}
u(x) = \frac{1}{2}x^{\top}Ax  + b \cdot x + c + O(|x|^{2-n}) \quad \text{as } x \to \infty.
\end{equation}
\end{itemize} 
Our estimate \eqref{eq:global decay} recovers \eqref{eq:asympotic at infty n3} if $|\M u-\LL|(E_A(x,\rho))=O(|x|^{\frac{(2-n)n}{2}})$ with $\rho = \tfrac{1}{2}\sqrt{x^{\top} A x}$, in particular, when $\M u-\LL$ is compactly supported. Barrier functions and comparison principles were used extensively in the proof of \eqref{eq:asympotic at infty n3} in \cite{caffarelli2003extension}, but they are not well suited to our measure-theoretic framework \eqref{eq:global-ma-eq-mu} and typically do not yield sharp constants. Instead, the sharp Alexandrov estimates developed in our earlier work  \cite{jin2025abp}  play a central role here. The asymptotic behavior of the solutions to $\det D^2 u = 1$ outside a bounded convex set was established earlier for $n=2$ by Ferrer-Mart\'inez-Mil\'an \cite{ferrer1999extension, ferrer2000space}. The existence and uniqueness of solutions  to \eqref{eq:global-ma-eq-mu}  satisfying the above asymptotic condition in this absolutely continuous setting was shown in Caffarelli-Li \cite{caffarelli2003extension} for $n \geq 3$ and in Bao-Xiong-Zhou \cite{bao2019existence} for $n \geq 2$. Chou-Wang \cite{chou1996entire} and Jian-Wang \cite{jian2014existence} proved existence of infinitely many entire solutions to general  Monge-Amp\`ere equations.

Within the measure-theoretic framework, we establish existence and uniqueness of entire convex solutions to \eqref{eq:global-ma-eq-mu} with a prescribed quadratic asymptote. 
\begin{Theorem}\label{thm:abp global existence} 
Suppose $n \geq 3$ and $1+\mu$ is a non-negative Borel measure satisfying $0<|\mu|(\mathbb{R}^n)<\infty$.  Then for every $A \in \A_n$, $b \in \mathbb{R}^n$, and $c \in \mathbb{R}$, there exists a unique convex solution of \eqref{eq:global-ma-eq-mu} satisfying \eqref{eq:asymptotic-quadratic}.

\end{Theorem}

A notable special case is when $\mu$ is a positive Dirac measure, leading to the isolated singularity problem
\begin{equation*}\label{eq:isolated singularity problem}
\det D^2 u = 1 + \omega_n a^n \delta_0 \quad \text{in } \mathbb{R}^n,
\end{equation*}
where $a$ is defined in \eqref{eq:total-variation}. Under unimodular affine transformation and addition of linear functions, the solutions are uniquely given by 
\begin{equation}\label{eq:isolated global solution}
W_a(x) = \int_0^{|x|} \left( r^n + a^n \right)^{\frac{1}{n}} \ud r.
\end{equation}
This classification result was proved by Jörgens~\cite{jorgens1955harmonische} for $n=2$ and by Jin-Xiong~\cite{jin2016solutions} for $n \ge 3$.
Via Legendre duality, it implies that every convex solution of the hyperplane obstacle problem\begin{equation*}\label{eq:obstacle problem}
\det D^2 v = \chi_{\{v > 0\}}, \quad v \geq 0\quad  \text{in } \mathbb{R}^n,
\end{equation*}
(where $\chi_E$ denotes the characteristic function of a set $E$) such that the set $\{v > 0\}$ is bounded and has positive Lebesgue measure,  is affine equivalent to the Legendre transform $W_a^*$ of $W_a$, namely
\begin{equation}\label{eq:obstacle global solution}
W_a^*(x) = \int_{0}^{|x|} \max\left\{ r^n - a^n,\, 0 \right\}^{\frac{1}{n}}   \ud r,
\end{equation}
with $|\{ W_a^* = 0\}| = \omega_n a^n$. Furthermore, 
\begin{equation*}
\lim_{|x|\to \infty} \left(W_1(x) - \frac{1}{2}|x|^2\right)=\lim_{|x|\to \infty} \left( \frac{1}{2}|x|^2-W_1^*(x)\right)=  \frac{\Gamma\left(\frac{1}{n}\right)\Gamma\left(\frac{n-2}{n}\right)}{2n\Gamma\left(\frac{n-1}{n}\right)}=d_{n,0}.
\end{equation*}

We also consider entire solutions with multiple, possibly infinitely many, isolated singularities. 
In dimension $n=2$, solutions are known to be strictly convex and admit an explicit representation \cite{galvez2005space}. For $n \ge 3$, the global structure is subtler. Mooney-Rakshit \cite{mooney2021solutions,mooney2023singular} constructed solutions that fail to be strictly  convex by developing linear segments or polyhedral flats of dimension less than $n/2$. Using the global Alexandrov estimate \eqref{eq:abp-global}, we obtain a quantitative sufficient condition ensuring strict convexity, and hence smoothness away from the set of the isolated singularities.
\begin{Theorem}\label{prop:strictly convex cond global}
Let $n \geq 3$. Suppose $u$ is an entire convex solution of 
\begin{equation}\label{eq:sing global infinite sing}
\det D^2 u = 1 +\omega_n  \sum_{i=1}^{\infty}  a_i^n \,\delta_{y_i}
\quad \text{in } \mathbb{R}^n,
\end{equation}
where $a_i \geq  0$ satisfies $0<\sum_{i=1}^{\infty}\omega_n a_i^n <\infty$ and  $|y_i|\to\infty$ as $i\to\infty$. Assume that $u$ asymptotically approaches $\frac{1}{2} x^\top A x + b \cdot x + c$ for some $A \in \A_n$, $b \in \mathbb{R}^n$, and $c \in \mathbb{R}$ in the sense of \eqref{eq:asymptotic-quadratic}. If
\begin{equation}\label{eq:sufficient cond for sc}
\sum_{i=1}^{\infty} a_i^n
< (8 d_{n,0})^{\frac{n}{2}} \inf_{i \neq j} \big| (y_i - y_j)^{\top}A (y_i - y_j) \big|^\frac{n}{2},
\end{equation}
then $u$ is strictly convex, and smooth in $\mathbb{R}^n \setminus \{y_1, y_2, \dots\}$.
\end{Theorem}

The constant $(8 d_{n,0})^{\frac{n}{2}} $ in Theorem \ref{prop:strictly convex cond global} is not expected to be optimal.  If the solutions to \eqref{eq:sing global infinite sing} are strictly convex, then they are smooth away from the isolated singularities; see Caffarelli \cite{caffarelli1990ilocalization,caffarelli1990interiorw2p,caffarelli1991regularity}. 
The regularity of the tangent cones at the isolated singularities is also well understood: see G\'alvez-Jim\'enez-Mira  \cite{galvez2015classification} for $n=2$,  and Savin \cite{savin2005obstacle} and Huang-Tang-Wang \cite{huang2024regularity} for general $n$.

The structure of the paper is as follows. 
Section~\ref{sec:abp estimates} revisits the extremal configurations to the Alexandrov estimate in bounded domains. 
In Section~\ref{sec:decay  estimates}, we obtain decay estimates for the solutions. 
Finally, Section~\ref{sec:main results proof} presents the proofs of our main results.

\medskip 

\noindent \textbf{Conflict of interest:} All authors certify that there is no actual or potential conflict of interest about this article.

\section{Alexandrov estimates}\label{sec:abp estimates}

Let  $\Omega \subset \mathbb{R}^n$ be a bounded convex domain. For two convex functions  $u,\varphi\in C(\overline\Omega)$ with $u = \varphi$ on $\partial \Omega$, the Alexandrov estimate (see, e.g., (15) in \cite{jin2025abp}) states that
\begin{equation} \label{eq:abp gene 1/n}
\|u - \varphi\|_{L^{\infty}(\Omega)} \leq C(n) a |\Omega|^{\frac{1}{n}},
\end{equation}
where  $a \geq 0$ is defined by
\begin{equation}\label{defn:a}
\omega_n a^n = |\M u - \M \varphi|(\Omega).
\end{equation} 
The order $a$ in \eqref{eq:abp gene 1/n} is optimal  as $a \to \infty$; the case $\varphi \equiv 0$ reduces to the classical Alexandrov maximum principle.

For a convex function $\varphi\in C(\overline{\Omega})$ and a constant $a\ge 0$, we define the class
\begin{equation*}\label{eq:abp class Da}
\D_{a, \varphi}  = \left\{ w \in C(\overline\Omega):    w  \text{ is convex},\   w=\varphi \text{ on } \partial \Omega, \ \mbox{and } |\M w - \M \varphi|(\Omega)\leq \omega_n a^n \right\}.
\end{equation*}
This definition is symmetric: $w \in \D_{a,\varphi}$ if and only if $\varphi \in \D_{a,w}$. Let us consider the following two specific families of functions in $\D_{a, \varphi} $:
\begin{itemize}
    \item[(i).] solutions to the isolated singularity problem
\begin{equation}\label{eq:defnuay}
\M u_{a}(\cdot, y) =\M\varphi+ \omega_na^n\delta_y  \quad\text{in }   \Omega, \quad u_{a}(\cdot, y) =\varphi \quad\text{on } \partial \Omega,
\end{equation}
where $\delta_y$ represents the Dirac measure centered at $y \in \Omega$. The graphs of these two functions can be illustrated as follows:

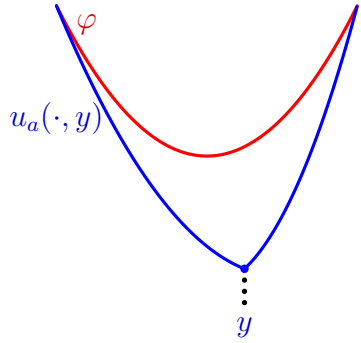
\begin{center}
\begin{tikzpicture}[scale=1, line cap=round, line join=round]
  \draw[very thick, red]
    plot[domain=-2:2, samples=200] (\x,{0.5*\x*\x});
  \node[red!80!black] at (-1.6,1.75) {$\varphi$};
\draw[very thick, blue]
    plot[domain=-2:0.5, samples=300] (\x,{0.4*(\x - 1)*(\x - 1) - 1.6});
\draw[very thick, blue]
  plot[domain=0.5:2, samples=300] (\x,{(14.0/15.0)*\x*\x - (26.0/15.0)});
  \node[blue!80!black] at (-2,0.5) {$u_a(\cdot,y)$};
\fill[blue] (0.5,-1.5) circle (1.6pt);
\foreach \k in {1,2,3} {
  \fill[black] (0.5,{-1.5 - 0.15*\k}) circle (1pt);
}
\node[blue!80!black, anchor=north] at (0.5,{-1.5 - 0.5}) {$y$};
\end{tikzpicture}
\end{center}

    \item[(ii).] solutions to the (hyperplane) obstacle problem taking the form\footnote{When $\M\varphi$ has unbounded density, the first equation in \eqref{eq:defnva} is interpreted as
\begin{equation*}\label{eq:defnvaq 1}
\M v_{a}(\cdot, p) = \M v_{a}(\cdot, p)  \quad \text{on } \left\{ v_{a}(\cdot, p) >  \ell_{p}+h_{a,p}\right\},\quad  
\M v_{a}(\cdot, p) \leq \M \varphi .
\end{equation*}}
\begin{equation}\label{eq:defnva}
\M v_{a}(\cdot, p) =\M\varphi \cdot \chi_{\left\{ v_{a}(\cdot, p)> \ell_{p}+h_{a,p}\right\}}   \quad\text{in }   \Omega, \quad v_{a}(\cdot, p) =\varphi\quad\text{on } \partial \Omega ,
\end{equation}
where 
\begin{itemize}
\item $p \in \partial \varphi(\Omega)$ is a subgradient, 
\item $\ell_p(x) = \varphi(x_p) + p\cdot(x-x_p)$ is a support function of $\varphi$ at some $x_p \in \Omega$,  
\item the parameter $h_{a,p} \in [0, \inf_{\partial \Omega } (\varphi-\ell_p)]$  is chosen such that either $h_{a,p} = \inf_{\partial \Omega} (\varphi - \ell_p)$, or the corresponding solution $v_a(\cdot, p)$ satisfies
\begin{equation}\label{eq:defnvaq 2}
\omega_n a^n=|\M v_{a}(\cdot, p) - \M \varphi|(\Omega),
\end{equation}
\item $\chi_E$ is the characteristic function of a set $E$.  
\end{itemize}
The obstacle for \eqref{eq:defnva} is given by $\ell_p + h_{a,p}$, and the set $\{x:\;v_{a}(x, p) = \ell_p + h_{a,p}\}$ is the coincidence set. The graph of the function $v_a(\cdot,p)$ can be illustrated as follows: 
\begin{center}
\begin{tikzpicture}[scale=0.9, line cap=round, line join=round]
  \draw[very thick, red]
    plot[domain=-3:3, samples=200] (\x,{0.5*\x*\x});

  \node[red!80!black] at (-2.6,2.75) {$\varphi$};

  \def\xp{0.5}
  \pgfmathsetmacro{\yp}{0.5*\xp*\xp}  
  \pgfmathsetmacro{\m}{\xp}           

  \draw[very thick, red]
    plot[domain=-1:3, samples=2] (\x,{\yp + \m*(\x-\xp)});
  \node[red!80!black, below] at (2.7,{\yp + \m*(2.7-\xp) - 0.05}) {$\ell_p$};

  \fill[red!80!black] (\xp,\yp) circle (2pt);
  \node[red!80!black, anchor=north west, xshift=2pt, yshift=4pt]
        at (\xp,\yp) {\scriptsize $(x_p,\;\varphi(x_p))$};

  \def\hoffset{1.5} 
  \draw[very thick, red]
    plot[domain=-2:3, samples=2] (\x,{\yp + \m*(\x-\xp) + \hoffset});
  \node[red!80!black, anchor=north west, xshift=-5pt] at (2.7,{\yp + \m*(2.7-\xp) + \hoffset - 0.05})
        { $\ell_p + h_{a,p}$};
        
  \draw[very thick, blue]
    plot[domain=-3:0, samples=200] (\x,{(25.0/72.0)*\x*\x + 11/8.0});
  \node[blue!80!black, anchor=north west, xshift=7pt]
        at (-2.9,{(25.0/72.0)*(-2.9)*(-2.9) + 11/8.0})
        {$v_a(\cdot, p)$};
  \draw[very thick, blue]
    plot[domain=1:3, samples=200] (\x,{(1.0/3.0)*\x*\x - (1.0/48.0)*\x + 25/16.0});
  \node[blue!80!black, anchor=west, xshift=3pt]
        at (3,{(1.0/3.0)*3*3 - (1.0/48.0)*3 + 25/16.0})
        {};

  \draw[very thick, blue]
    plot[domain=0:1, samples=2] (\x,{0.5*\x + 11/8.0});
  \node[blue!80!black, anchor=west, xshift=2pt]
        at (1,{0.5*1 + 11/8.0})
        { };
        
\end{tikzpicture}
\end{center}
\end{itemize}  
 
\begin{Theorem}[{\cite[Theorem 3.1]{jin2025abp}}]\label{thm:abp extremal}
Let $a\ge 0$ be a constant, and $\varphi\in C(\overline{\Omega})$ be a convex function. 
For each $y \in \Omega$, we have
\begin{equation*}\label{eq:pointwise ext lower}
\inf_{w \in \D_{a, \varphi} } w(y )=u_a(y ,y ),
\end{equation*}
and
\begin{equation*}\label{eq:pointwise ext upper}
\sup_{w\in \D_{a, \varphi} } w(y )=\sup_{p \in \partial \varphi (\Omega)}v_{a}(y ,p),
\end{equation*}
where the supremum on the right-hand side of \eqref{eq:pointwise ext upper} is actually achieved by some $v_{a}(\cdot,p)$ having $y $ in its coincidence set.
\end{Theorem}

Using this theorem, we obtained extremal Alexandrov estimates in \cite{jin2025abp} under the nondegeneracy condition
\begin{equation}\label{eq:nondegeneracy}
0<\lambda \le \det D^{2}\varphi \le \Lambda<\infty.
\end{equation}
In particular, for $n\ge 3$, Theorem 1.1 in \cite{jin2025abp}, when stated on a section of $\varphi$, implies the following affine invariant estimate.

\begin{Theorem}\label{thm:ordera2}
Let $n\ge 3$, $\epsilon\in(0,1)$, $\Omega\subset \mathbb{R}^{n}$ be a bounded convex domain, and  $\varphi\in C(\overline{\Omega})$ be a convex function satisfying \eqref{eq:nondegeneracy} in $\Omega$ and $\varphi=0$ on $\partial\Omega$. 
For $t\in(0,1)$, define the sublevel sets
\[
\Omega_{t}:=\{x\in\Omega:\ \varphi(x)<(1-t) \inf_{\Omega}\varphi\}.
\]
Then there exists a constant $C(n,\epsilon,\lambda,\Lambda)>0$, depending only on $n,\epsilon,\lambda,$ and $\Lambda$, such that for every $t\in(0,1-\epsilon)$ and every convex function $u\in C(\overline{\Omega_{t}})$ satisfying $u=\varphi$ on $\partial\Omega_{t}$, one has
\[
\left\|u-\varphi\right\|_{L^{\infty}(\Omega_t)}
\le
C(n,\epsilon,\lambda,\Lambda)\Bigl(\omega_{n}^{-1}\,|\M u-\M \varphi|(\Omega_{t})\Bigr)^{\frac{2}{n}}.
\]
\end{Theorem}
Therefore, part (ii) of Theorem \ref{thm:abp global rigidity} can be viewed as a global version of Theorem \ref{thm:ordera2} with a sharp constant.

Under the regularity assumptions $\partial \Omega \in C^{2,\alpha}$ and $\varphi \in \C_+^{2,\alpha}(\overline{\Omega})$,
where $\alpha \in (0,1)$ and $\C_+^{2,\alpha}$ denotes the space of $C^{2,\alpha}$ functions with \emph{positive definite Hessians}, we in 
\cite[(49) and (58)]{jin2025abp} established sharp pointwise bounds for the extremal functions. 
For simplicity, assuming $n\geq 3$, $\varphi(0)=0$, $\nabla \varphi(0)= 0$, and $\det D^2 \varphi(0)=1$, then we obtained:
\begin{equation}\label{eq:dna2 uaap}
\varphi(x) \geq u_{a}(x,0)  \geq \varphi(x) -\min\left\{C \frac{ a^{n}}{\abs{x}^{n-2}},\; d_{n,0}a^2 + C a^{2+\beta} \right\} \quad  \text{in } \Omega
\end{equation}
and
\begin{equation}\label{eq:dna2 obs ap x}
\varphi(x) \leq v_{a}(x,0)  \leq \varphi(x) + \min\left\{ C\frac{ a^{n}}{\abs{x}^{n-2}},\; d_{n,0}a^2 + C a^{2+\beta} \right\} \quad \text{in } \Omega,
\end{equation} 
where $\beta = \frac{(n-2)\alpha}{n+\alpha}$, where $C> 0$ depends only on $n$, $\alpha$, $\operatorname{dist}(0,\partial\Omega)$,  $\operatorname{diam}(\Omega)$, $\left\|\partial \Omega\right\|_{C^{2,\alpha}}$, $\left\|D^2 \varphi\right\|_{C^{\alpha}(\overline{\Omega})}$ and $\left\|(D^2 \varphi)^{-1}\right\|_{L^\infty(\Omega)}$.
Moreover,

\begin{Theorem}[{\cite[Theorem 1.2  for $n\ge 3$]{jin2025abp}}]\label{thm:ordera2 more}
There exists $C>0$ (dependence specified below) such that for every convex $u \in C(\overline{\Omega})$ with $u=\varphi$ on $\partial \Omega$ and for every $x_0 \in \Omega$, we have: 
\begin{equation*}\label{eq:perturb-result}
- \lambda_0^{-1}   - C a^{ \beta}   \leq \frac{u(x_0) - \varphi(x_0)}{d_{n,0}a^2} \leq \lambda_0  + C a^{ \beta} ,  
\end{equation*} 
where $\lambda_0= (\det D^2 \varphi(x_0))^{1/n}$, $a\geq 0$ is as in \eqref{defn:a}, $\beta = \frac{(n-2)\alpha}{n+\alpha}$ and $d_{n,0}$ is given by \eqref{eq:def dn00}.
\end{Theorem}

\begin{Theorem}[{\cite[Theorem 1.3]{jin2025abp}}]\label{thm:rigidity}
There exists $\varepsilon_0>0$  such that for every $x_0\in\Omega$, every $\rho \in (0,1)$ and every convex $u \in C(\overline{\Omega})$ with $u = \varphi$ on $\partial \Omega$ satisfying $a \leq \varepsilon_0$ (where $a$ is defined in \eqref{defn:a}), there exist positive constants $c_\rho$ and  $C_{\rho}$ such that 
\begin{align*}
\frac{ u(x_0)-\varphi(x_0)}{d_{n,0}a^2} 
& \geq -\lambda_{0}^{-1}-C_{\rho}a^{\beta}+c_{\rho}\left(\frac{\omega_na^n-\mu \left( E_A\left(x_0, \rho a\right) \cap\Omega\right) }{\omega_na^n} \right)^\frac{n}{2},\\
\frac{ u(x_0)-\varphi(x_0)}{d_{n,0}a^2} 
& \leq \lambda_{0} +C_{\rho}a^{\beta}-c_{\rho}\left(\frac{\omega_na^n+\mu \left( E_A(x_0, (1+\rho)a)\cap \Omega\right)}{ \omega_na^n} \right)^n, 
\end{align*}
where $\lambda_0,\beta,d_{n,0}$ are as in Theorem \ref{thm:ordera2 more}, $A= D^2 \varphi(x_0)$, and $E_A(x, r)$ is defined in Theorem \ref{thm:global decay}.
\end{Theorem}

The constants $C$ in Theorem \ref{thm:ordera2 more} and $\varepsilon_0$ in Theorem \ref{thm:rigidity} depend only on $n$, $\alpha$, $\operatorname{diam}(\Omega)$, $\left\|\partial \Omega\right\|_{C^{2,\alpha}}$, $\left\|D^2 \varphi\right\|_{C^{\alpha}(\overline{\Omega})}$ and $\left\|(D^2 \varphi)^{-1}\right\|_{L^\infty(\Omega)}$. The constants $c_\rho$ and $C_\rho$ Theorem \ref{thm:rigidity} depend only on the same parameters, and additionally on $\rho$.

\section{A quantitative decay estimate}\label{sec:decay estimates}

We will use the following two lemmas in our earlier work \cite{jin2025abp}.
\begin{Lemma}[{\cite[Lemma 3.9]{jin2025abp}}]\label{lem:comparison gene} 
Let $u, \varphi \in C(\overline{\Omega})$ be convex functions, and let $W\subset \Om$ be a closed set such that
\[
\M u(E) \leq \M \varphi(E), \quad \text{ for every Borel set } E \text{ satisfying } W\subset E \subset \Om.
\] 
Then
\[
\max_{x \in W}\{u(x)-\varphi(x)\} \geq \min _{x \in \partial \Omega}\{u(x)-\varphi (x)\} .
\]	
\end{Lemma}

\begin{Lemma}[A boundary contact lemma, {\cite[Lemma 6.1]{jin2025abp}}]\label{lem:extends to boundary}
Let $u, \varphi \in C(\overline{\Omega})$ be convex functions with $u = \varphi$ on $\partial \Omega$. Suppose that near $\partial \Omega$,  $\varphi$ is strictly convex and satisfies $0 < \lambda \leq \det D^2 \varphi \leq \Lambda < \infty$, and $\varphi$ is
not identically equal to $u$ there. If the measure $\mu = \M u - \M \varphi$ is compactly supported in $\Omega$ and satisfies $\mu(\Omega) = 0$, then both sets $\overline{\{u <\varphi\}}$ and $\overline{\{ \varphi<u\}}$ intersect the boundary $\partial\Omega$.
\end{Lemma}

The following is a comparison principle for the obstacle problem. 
\begin{Lemma}\label{lem:comparison principle equal volume}
Let $\Omega\subset \mathbb{R}^{n}$ be a bounded convex domain. Let
$w\in C(\overline{\Omega})$ and $\tilde w\in C(\overline{\Omega})$ be solutions of the obstacle problems
\[
\det D^2 v = \chi_{\left\{ v>  t\right\}}   \quad\text{in }   \Omega, \quad v\geq t,
\]
where the obstacle $t$ is taken to be the constants $h $ and $\tilde h $, respectively.
Set $K=\{w =h\}$ and $\widetilde{K}=\{ \tilde{w}=\tilde{h}\}$.
Assume that $w\ge \tilde w$ on $\partial\Omega$. If the coincidence sets satisfy $K\neq\emptyset$ and
\[
|K|\geq|\widetilde K|,
\]
then we have
\[
h \geq \tilde{h} \quad \text{and} \quad  w \geq \tilde{w}.
\]
\end{Lemma}
\begin{proof}
Suppose, for contradiction, that $\tilde h>h$. Then Lemma~2.3 in~\cite{jin2025regularity} yields
\[
\tilde w-\tilde h\le w-h\quad \text{in }\Omega.
\]
In particular, $K\subset \widetilde K$, and hence $\M w\ge \M \tilde w$.
Since $|K|\geq |\widetilde K|$, we in fact have $\M w=\M \tilde w$.
By the comparison principle, this implies $w\ge \tilde w\ge \tilde h$ in $\Omega$. Since $w=h$ on the nonempty set $K$, we obtain $h\ge \tilde h$, contradicting $\tilde h>h$.
Therefore, $h\ge \tilde h$.

Finally, note that on $\widetilde K=\{\tilde w=\tilde h\}$ we have
\[
w\ge h\ge \tilde h=\tilde w.
\]
Applying the comparison principle in $\Omega\setminus \widetilde K$ then yields $w\ge \tilde w$ in $\Omega$.
\end{proof}

We recall that for a $C^1$ convex function $\varphi$ on a convex domain $\Omega$, the section of $\varphi$ at $x_0 \in \Omega$ with height $t > 0$ is the sub-level set
\[
S_{t}^{\varphi}(x_0)= S_{t,\nabla \varphi(x_0)}^{\varphi}(x_0) = \left\{ x  \in \overline{\Omega}  :\;  \varphi (x) < \varphi (x_0)+\nabla\varphi(x_0) \cdot (x-x_0)+t  \right\}.
\]
Moreover, if $\det D^2 \varphi = 1$ and $S_t^\varphi(x_0) \subset \Omega$, then it follows from Caffarelli's regularity theory \cite{caffarelli1990ilocalization,caffarelli1990interiorw2p,caffarelli1991regularity} that there exists an affine transformation $A \in \mathcal{A}_n$ such that
\[
 A B_{t^{1/2}/C_0} \subset  (S_{t}^{\varphi}(x_0) -x_0) \subset A B_{C_0t^{1/2}}
\]
for some positive constant $C_0$ depending only on $n$.

The following localized (compared to Theorem \ref{thm:ordera2 more}) Alexandrov estimate plays a key role in our analysis. For two non-negative quantities $ b_1$ and $b_2 $, we adopt the notation 
\[
b_1 \approx b_2 
\] 
to indicate that $ b_1 \leq Cb_2 $ and $b_2 \leq Cb_1 $  for some constant $C$ depending only on $n$. 

\begin{Theorem}\label{prop:decay estimate local}
Let $n \ge 3$ and let $\Omega \subset B_{4C_0^2} \subset \R^n$ be a convex domain. 
Let $\varphi \in C(\overline{\Omega})$ be a convex function  satisfying
\[
\det D^2 \varphi = 1 \quad \text{in } \Omega, \quad B_{1}  \subset  \{ \varphi <0\} \subset \Omega,\quad \inf_{\Omega} \varphi = \varphi(0)  .
\] 
Let $u \in C(\overline{\Omega})$ be a convex function satisfying $u = \varphi$ on $\partial \Omega$, and set $\mu = \M u - \M \varphi$. Fix $\alpha=\frac{1}{2}$ and $\beta = \frac{(n-2)\alpha}{n+\alpha}$.
Then for any $0 < r \le 1/64$ and point $x_0 \in B_{1/64}$, we have
\begin{equation}\label{eq:asymptotic-quadratic-order local}
\left|u(x_0) -\varphi(x_0)\right| \leq d_{n,0}b^2 + C(n)  \left(\frac{a^n}{r^{n-2}}+\frac{b^{2+\beta}}{r^{\beta}}\right),
\end{equation} 
where the constants $a \ge b \ge 0$ are given by 
\[
|\mu|(\Omega) = \omega_n a^n, \quad |\mu|\bigl(B_{r}(x_0)\bigr)+|\mu|\bigl(\Omega\setminus B_{1/32}\bigr) = \omega_n b^n.
\] 

In the special case where $\varphi$ is quadratic, the estimate \eqref{eq:asymptotic-quadratic-order local} can be strengthened to
\begin{equation}\label{eq:asymptotic-quadratic-order local special}
\left|u(x_0) -\varphi(x_0)\right| \leq d_{n,0}b^2 + C(n)  \frac{a^n}{r^{n-2}}.
\end{equation} 
\end{Theorem}

\begin{proof}
\textbf{Step 1. Sign reduction.} 

Let us express
\[ 
\mu=\mu_+-\mu_-,
\]
where $\mu_+$ and $\mu_-$ are non-negative measures and $\mu_+ \perp \mu_-$. 
Let $u_+$ and $u_-$ be the convex solutions to 
\begin{align*}
& \M u_{+}=\M \varphi+\mu_+  \quad \text{in } {\Omega}, \quad u_{+} =u\quad \text{on } \partial \Omega, \\
& \M   u_{-} =\M \varphi-\mu_- \quad \text{in } {\Omega}, \quad   u_{-}  =u\quad \text{on } \partial {\Omega},  
\end{align*} 
respectively. The comparison principle yields
\begin{equation*} \label{eq:lem-order}
u_{+} \leq u  \leq u_{-} \quad \mbox{and} \quad u_{+} \leq \varphi  \leq u_{-}.
\end{equation*}
Consequently,
\[
u_+ - \varphi\le u-\varphi\le u_- -\varphi.
\]
From estimates of $u_+ - \varphi$ and $u_- - \varphi$, we obtain bounds for $|u - \varphi|$.

Therefore, by this reduction, we may assume, without loss of generality, in the rest proof that $\mu$ is either nonnegative or nonpositive, which implies $u \leq \varphi$ or $u \geq \varphi$, respectively. 

\medskip

\textbf{Step 2. Proof of \eqref{eq:asymptotic-quadratic-order local} under the assumption $|\mu|\bigl(\Omega\setminus B_{1/32}\bigr) =0$.} 

In this step, we are going to prove \eqref{eq:asymptotic-quadratic-order local} under the additional assumption $|\mu|\bigl(\Omega\setminus B_{1/32}\bigr) =0$. This assumption will be removed in Step 3. Recall that we will assume that $\mu$ is either non‑negative or non‑positive.

\medskip

\textbf{Step 2.1. Regularity and smallness reduction.} 

Set $h_0 = -\varphi(0)$. Then by the comparison principle, $\frac{1}{2} \le h_0 \le C(n)$. Define the section of $\varphi$ at $0$ by $$S_{h}=\{ x \in \overline{\Omega}: \varphi (x) < -h_0+h\}.$$ Then it follows from \cite{caffarelli1990ilocalization,gutierrez2016monge},  that  $\varphi \in \C_+^{2,\alpha}(\overline{S_{h_0/2}})$, $\partial S_{h_0/2} \in C^{2,\alpha}$ and $$\|D^2\varphi \|_{C^{\alpha}(\overline{S_{h_0/2}})} +\|(D^2\varphi)^{-1}\|_{L^{\infty}({S_{h_0/2}})} + \|\partial S_{h_0/2}\|_{C^{2,\alpha}}\le C(n)$$ for some constant $C(n)>0$ depending only on the dimension.

Consider the solution of
\[
\M \tilde{u}=\M u=\M \varphi+\mu   \quad \text{in } {S_{h_0/2}}, \quad \tilde u =\varphi =-\frac{h_0}{2}\quad \text{on } \partial {S_{h_0/2}}.
\]
Note that  $B_{1/2}\subset  \{\varphi<0\}= S_{h_0}$, we have 
$$
\operatorname{supp}\mu \subset B_{1/32}\subset \frac{1}{16} S_{h_0} \subset S_{h_0/16}.
$$ 
Hence, an application of \cite[Lemma 4.1]{jin2025abp} to the functions $u(C_0^2x)/h_0$ and $\varphi(C_0^2x)/h_0$ yields
\[
|u-\varphi| \leq C(n)|\mu|(\Omega) \leq C(n)a^n \quad \text{in } \Omega\setminus S_{h_0/2}.
\]
Thus, the comparison principle yields
\[
\|u-\tilde{u}\|_{L^{\infty}(S_{h_0/2})} \leq \|u-\tilde{u}\|_{L^{\infty}(\partial S_{h_0/2})} = \|u-\varphi\|_{L^{\infty}(\partial S_{h_0/2})} \leq C(n)a^n.
\]
And it suffices to establish that 
\[
|\tilde{u}(x_0)-\varphi(x_0)|  \leq d_{n,0}b^2 + C(n)  \left(\frac{a^n}{r^{n-2}}+\frac{b^{2+\beta}}{r^\beta}\right).
\]
We perform the scaling
\[
\tilde{u}(x) \;\mapsto\; 2\tilde{u}(2^{-\frac{1}{2}}x),
\qquad
\varphi(x) \;\mapsto\; 2\varphi(2^{-\frac{1}{2}}x),
\qquad
S_{h_0/2} \;\mapsto\; 2 S_{h_0/2},
\]
and relabel the transformed objects as $u$, $\varphi$, and $\Omega$. Under this normalization, it suffices to prove  \eqref{eq:asymptotic-quadratic-order local} assuming $\partial \Omega \in C^{2,\alpha}$, $\varphi \in \C_+^{2,\alpha}(\overline{\Omega})$, 
$$
\left\|\partial \Omega\right\|_{C^{2,\alpha}}+\left\|D^2\varphi\right\|_{C^{\alpha}(\overline{\Omega})}+\left\|(D^2\varphi)^{-1}\right\|_{L^{\infty}({\Omega})}\le C(n),
$$ 
and $B_{1} \subset \Omega =\{\varphi<0\}\subset B_{8C_0^2}$.
Now we can apply Theorem~\ref{thm:ordera2 more} to obtain
\[
\|u-\varphi\|_{L^{\infty}(\Omega)} \le d_{n,0} a^2+ Ca^{2+\beta}.
\]
Hence, the estimate \eqref{eq:asymptotic-quadratic-order local} is nontrivial only when the ratio $a/r$ is small. Since $r\leq 1/64$, it follows that $a$ is also small. Therefore, we will further assume, without loss of generality, $a/r$ and $a$ are small.

\medskip

\textbf{Step 2.2. Proof of \eqref{eq:asymptotic-quadratic-order local} for non-negative $\mu$ assuming $|\mu|\bigl(\Omega\setminus B_{1/32}\bigr) =0$.} 

We consider the case $\mu$ is non-negative, which implies $u \leq \varphi$. By subtracting a linear function and translation, we only need to prove for $x_0 = 0$.

Let $\hat{\varphi} \leq \varphi$ be the solution of
\[
\M \hat{\varphi} =\M {u} \cdot \chi_{\Omega \setminus B_{r}(0)}+\M \varphi \cdot \chi_{B_{r}(0)} \quad \text{in }   \Omega, \quad \hat{\varphi}=\varphi \quad \text{on } \partial \Omega.
\]
Let $\hat{u} \leq \hat{\varphi}$ be the solution of
\begin{equation*}
\M \hat{u} =\M {u} \cdot \chi_{\Omega \setminus B_{r}(0)}+\M \varphi \cdot \chi_{B_{r}(0)}+\omega_nb^n\delta_0 \quad \text{in }   \Omega, \quad \hat{u} =\varphi \quad \text{on } \partial \Omega.
\end{equation*}
Then $u,\hat{u}\in \mathcal{D}_{b,\hat{\varphi}}$. Applying Theorem~\ref{thm:abp extremal} to the class $\mathcal{D}_{b,\hat{\varphi}}$, we obtain
\[
u(0) \geq \hat{u}(0).
\]
Applying Theorem~\ref{thm:ordera2 more} to $\frac{\hat u(rx)}{r^2}$ and $\frac{\varphi(rx)}{r^2}$, together with the comparison principle, implies
\begin{equation}\label{eq:theorem 2.3 lower general}
\left|\hat{u}(0) - \varphi(0)\right| \leq \| \hat{u} - \varphi \|_{L^{\infty}(\partial B_{r/2}(0))} + d_{n,0} b^2 + C(n) \frac{b^{2+\beta}}{r^{\beta}}.
\end{equation}
Thus, it suffices to show that 
\begin{equation}\label{eq:theorem 2.3 lower general 1}
\kappa:=\| \hat{u} - \varphi \|_{L^{\infty}(\partial B_{r/2}(0))}  \leq C(n) \frac{a^n}{r^{n-2}} .
\end{equation}

Applying  Theorem~\ref{thm:ordera2 more}, we find that 
\[
\|\hat{u}-\varphi\|_{L^{\infty}(\Omega)} \le d_{n,0} a^2+ Ca^{2+\beta}.
\]
Since $r \gg a$ and $\varphi \in \C_+^{2,\alpha}$, it follows from \cite{savin2007perturbation} (see also \cite{jian2007continuity}) that $\hat{u} \in C^{1,1}(B_{7r/8}(0) \setminus B_{r/8}(0))$, and that the linearized operator
\begin{equation*}\label{eq:linearized mae}
L[w] := A^{ij} \partial_{ij} w,
\end{equation*}
where the coefficient matrix $\{A^{ij}\} = \int_0^1 \operatorname{Cof}\left(t D^2 u + (1-t) D^2 \varphi\right) dt$ is defined via the cofactor matrix operator, is uniformly elliptic in $B_{7r/8}(0) \setminus B_{r/8}(0)$. Note that 
\[
L[u - \varphi] = \det D^2 u - \det D^2 \varphi = 0 \quad \text{in } B_{7r/8}(0) \setminus B_{r/8}(0).
\] 
The Harnack inequality then implies  
\begin{equation}\label{eq:harnackv}
\varphi-\hat{u}  \approx \kappa  \quad \text{in } B_{3r/4}(0) \setminus B_{r/4}(0).
\end{equation}

Let $\eta \in C^{\infty}([0,\infty))$ be a non-increasing function satisfying $\eta = 1$ on $[0,1/16]$ and $\eta = 0$ on $[1/4,\infty)$. 
Let $w \leq \varphi$ be the solution to
\begin{equation*}\label{eq:approximation 3.1p}
\M w = \M \varphi+\varepsilon\eta \left(\left|\frac{|x|}{r}-\frac{1}{2}\right|\right)\quad \text{in }   \Omega, 
\quad w =\varphi \quad \text{on } \partial \Omega,
\end{equation*}
where $\varepsilon > 0$ is chosen such that $\M  w(\Omega) = \M \varphi(\Omega) + \omega_n a^n$, which implies $\varepsilon \approx  a^n / r^{n}$ and $\M \hat{u}(E) \leq \M w(E)$, for every Borel set  $E$ satisfying $(B_{3r/4}(0) \setminus B_{r/4}(0))\subset E \subset \Om.$
By the comparison principle (Lemma~\ref{lem:comparison gene}), it follows that $\hat{u} \ge w$, and thus, $\varphi-\hat{u}\le \varphi-w$, at some point in the annulus $B_{3r/4}(0) \setminus B_{r/4}(0)$. Consequently, because of \eqref{eq:harnackv}, it suffices to prove that
\[
\|\varphi-w\|_{L^{\infty}(B_{3r/4}(0)\setminus B_{r/4}(0))}
\leq C(n) \frac{a^n}{r^{n-2}} .
\]

Let $u_a := u_a(\cdot,0)$ be the solution to \eqref{eq:defnuay} at $y=0$. Applying the comparison principle, we find that every connected component of $\{u_a < w\}$ must contain $0$; hence, $\{u_a < w\} $ is connected. We claim that there exists $z \in \partial B_{2r}(0)$ such that $u_a(z) \leq w(z)$. 
Suppose, for contradiction, that $\partial B_{2r}(0) \subset \{u_a > w\}$. Hence, $\{u_a <w\}\subset B_{2r}(0)$, that is, $\{u_a \ge w\}\subset \Omega\setminus B_{2r}(0)$.
Noting that $ \M w (\Omega) = \M u_a(\Omega) = \M \varphi(\Omega) + \omega_na^n$, Lemma~\ref{lem:extends to boundary} then implies that $u_a=w$ near $\partial \Omega$.
By the strong maximum principle (see, e.g., \cite{jian2025strong}), we further deduce that $u_a=w$ in $\Omega\setminus B_{2r}(0)$, and consequently on  $\partial B_{2r}(0)$. This yields a contradiction.

In conclusion,  we obtain from \eqref{eq:dna2 uaap} that
\[
w(z) \geq u_a(z) \geq \varphi(z) - C(n) \frac{a^n}{r^{n-2}} \quad \text{at some point}\quad  z \in \partial B_{2r}(0).
\]
Note that $w \in C^{2,\alpha}( \Omega)$.
By the uniform ellipticity of the linearized operator between $\varphi$ and $w$, the Harnack inequality then implies
\[
0\leq \varphi -w\leq C(n) (\varphi(z)-w(z))\leq  C(n) \frac{a^n}{r^{n-2}}  \quad \text{on } \partial B_{2r}(0).
\]
Note that we also have $\frac{\varphi(r x)}{r^2}, \frac{w(r x)}{r^2} \in C^{2,\alpha}(B_2(0))$. By considering  the linearized equation of $\frac{(w-\varphi)(r x)}{r^2\varepsilon}$, we obtain that
\[
\|w-\varphi \|_{L^{\infty}(B_{2r}(0))}  \leq 
\|w-\varphi \|_{L^{\infty}(\partial B_{2r}(0))}+C(n) \varepsilon r^2 \leq C(n)  \frac{a^n}{r^{n-2}} .
\]
This completes the proof for the non-negative case.

\medskip

\textbf{Step 2.3. Proof of \eqref{eq:asymptotic-quadratic-order local} for non-positive $\mu$ assuming $|\mu|\bigl(\Omega\setminus B_{1/32}\bigr) =0$.}

We now consider the case $\mu$ is non-positive, which implies $u \ge \varphi$. By subtracting a linear function and translation, we can assume $x_0 = 0$ and $\nabla \varphi(x_0)=0$. The proof below is analogous to that of the non-negative case. 

Let $\hat{\varphi} \geq \varphi$ be the solution of
\[
\M \hat{\varphi} =\M {u} \cdot \chi_{\Omega \setminus B_{r}(0)}+\M \varphi \cdot \chi_{B_{r}(0)} \quad \text{in }   \Omega, \quad \hat{\varphi}=\varphi \quad \text{on } \partial \Omega.
\]
Applying Theorem~\ref{thm:abp extremal} to the class $\mathcal{D}_{b,\hat{\varphi}}$, we obtain the existence of a subgradient $p \in \partial \hat{\varphi}(\Omega)$ and an obstacle solution $\hat{v} := \hat{v}_b(\cdot,p)$, whose coincidence set $\hat{K}$ contains $0$, such that
\[
u(0) \le \hat{v}(0).
\]
Here $\hat{v}_b(\cdot,p)  \ge \hat{\varphi}$ solves
\[
\M \hat{v}_b(\cdot,p)=\M \hat{\varphi} \cdot \chi_{\{\hat{v}_b(\cdot,p) > \hat{\ell}_p + \hat{h}_{b,p}\}} \quad \text{in } \Omega, \quad
\hat{v}_b(\cdot,p) = \hat{\varphi}
\quad \text{on } \partial\Omega,
\]
for some $\hat{h}_{b,p}>0$, where $\hat{\ell}_p$ denotes the supporting function to $\hat{\varphi}$ with slope $p \in \partial \hat{\varphi}(\Omega)$.
By Theorem~\ref{thm:ordera2 more} and the comparison principle, we have the estimate
\begin{equation}\label{eq:theorem 2.3 upper general}
\left|\hat{v}(0) - \varphi(0)\right| \leq \| \hat{v} - \varphi \|_{L^{\infty}(\partial B_{r/2}(0))} + d_{n,0} b^2 + C(n) \frac{b^{2+\beta}}{r^{\beta}}.
\end{equation}
Thus, it suffices to show that 
\begin{equation}\label{eq:theorem 2.3 upper general 1}
\kappa:=\| \hat{v} - \varphi \|_{L^{\infty}(\partial B_{r/2}(0))}  \leq C(n)  \frac{a^n}{r^{n-2}} .
\end{equation}

Applying Theorem~\ref{thm:ordera2 more}, we obtain
\[
\|\hat{v}-\varphi\|_{L^{\infty}(\Omega)} \le d_{n,0} a^2+ Ca^{2+\beta}.
\]
Since $\varphi\in \C_+^{2,\alpha}$, $0\in \hat{K}$, and $\hat{v}$ is linear on $\hat{K}$,
an elementary analysis gives
\[
|p|\leq Ca,\quad 
\hat{K}:=\left\{ \hat{v} = \hat{\ell}_{p}+\hat{h}_{b,p}\right\} \subset S_{Ca^2}^{\varphi} (0) \subset B_{Ca}(0)\subset B_{r/16}(0).
\] 
Therefore, $\det D^2 \hat{v} = \det D^2\hat{\varphi}=\det D^2 \varphi$ in $B_{r}(0)\setminus B_{r/16}(0)$, and it follows from \cite{savin2007perturbation} (see also \cite{jian2007continuity}) that $\hat{v} \in   C^{2,\alpha}(B_{7r/8}(0) \setminus B_{r/8}(0) )$.
Applying the Harnack inequality to the linearized equation between $\varphi$ and $\hat{v}$ then implies that
\begin{equation}\label{eq:harnacktildev}
\hat{v} -\varphi \approx \kappa  \quad \text{in } B_{3r/4}(0) \setminus B_{r/4}(0).
\end{equation}

Let $\tilde{w} \geq \varphi$ be the solution to
\begin{equation*}
\M \tilde{w} = \M \varphi-\varepsilon\eta \left(\left|\frac{|x|}{r}-\frac{1}{2}\right|\right)\quad \text{in }   \Omega, 
\quad \tilde{w} =\varphi \quad \text{on } \partial \Omega,
\end{equation*}
where $\varepsilon \approx  a^n / r^{n} > 0$ is chosen such that $\M  \tilde{w}(\Omega) = \M \varphi(\Omega) - \omega_n a^n$. 
By the comparison principle, as in Step 2, we deduce that $\hat{v} \leq w$ at some point in $B_{3r/4}(0) \setminus B_{r/4}(0)$. Consequently, because of \eqref{eq:harnacktildev}, it suffices to show that
\[
\|\tilde{w}-\varphi \|_{L^{\infty}(B_{3r/4}(0)\setminus B_{r/4}(0))}
\leq C(n) \frac{a^n}{r^{n-2}} .
\]

Let $v_a := v_a(\cdot,0)$ be the solution to \eqref{eq:defnva} with $p=0$, whose coincidence set $K_a$ is contained in $S_{Ca^2}^{\varphi}(0)\subset B_{r/16}(0)$, according to \cite[Lemma 3.11]{jin2025abp}.
Applying the comparison principle, we find that every connected component of $\{\tilde{w}< v_a\}$ must contain some point in $K_a \subset B_{r/16}(0)$.
Then, a parallel discussion to the non-negative case, combined with estimate \eqref{eq:dna2 obs ap x}, yields that, 
\[
\tilde{w}(z) \leq v_a(z) \leq \varphi(z) +C(n)\frac{a^n}{r^{n-2}} \quad \text{at some point} \quad  z \in \partial B_{2r}(0).
\]
which, together with the equation of $w$ and that $w \in C^{2,\alpha}( \Omega\setminus B_{r}(0))$, implying
\[
\tilde{w}-\varphi  \leq C(n) (\tilde{w}(z)-\varphi(z))\leq C(n)  \frac{a^n}{r^{n-2}} \quad \text{on } \partial B_{2r}(0).
\]
Note that we also have $\frac{\varphi(r x)}{r^2}, \frac{\tilde{w}(r x)}{r^2} \in C^{2,\alpha}(B_2(0))$. By considering  the linearized equation of $\frac{(\tilde{w}-\varphi)(r x)}{ \varepsilon r^2}$, we obtain that
\[
\|\tilde{w}-\varphi \|_{L^{\infty}(B_{2r}(0))}  \leq 
\|\tilde{w}-\varphi \|_{L^{\infty}(\partial B_{2r}(0))}+C(n) \varepsilon r^2 \leq C(n)  \frac{a^n}{r^{n-2}} .
\]
This completes the proof for the non-positive case.

\medskip

\textbf{Step 3. Proof of \eqref{eq:asymptotic-quadratic-order local}.} 

The Alexandrov estimate \eqref{eq:abp gene 1/n} implies that \eqref{eq:asymptotic-quadratic-order local} is nontrivial only  when $a$ is small.
Let $\tilde{\varphi}$ be the solution of
\[
\M \tilde{\varphi} =\M {\varphi} +\mu \cdot \chi_{B_{1/32}}\quad \text{in }   \Omega, \quad \tilde{\varphi}=\varphi \quad \text{on } \partial \Omega.
\]
Denote $|\mu|\bigl(\Omega\setminus B_{1/32}\bigr) = \omega_n \tilde{a}^n$ with $\tilde a \ge 0$. Then $u \in \D_{\tilde{a},\tilde{\varphi}}$. 
Applying Theorem~\ref{thm:abp extremal} to the class $\D_{\tilde{a},\tilde{\varphi}}$, we obtain 
\[
\tilde{u}_{\tilde{a}}(x_0 ,x_0)\leq u(x_0) \leq \tilde{v}_{\tilde{a}}(x_0 ,\tilde{p}).
\]
Here, $\tilde{u}_{\tilde{a}}(\cdot,x_0)$ solves $\M \tilde{u}_{\tilde{a}}(\cdot,x_0) =\M \tilde{\varphi} +\omega_n\tilde{a}^n \delta_{x_0}$ in $\Omega$ with $\tilde{u}_{\tilde{a}}(\cdot,x_0)=\tilde\varphi$,  and $\tilde{v}_{\tilde{a}}(\cdot,\tilde{p})$ solves the corresponding obstacle problem \eqref{eq:defnva} satisfying \eqref{eq:defnvaq 2} with $\tilde\varphi$ and $\tilde a$, whose coincidence set $\tilde{K}_{\tilde{a}}$ contains $x_0$, as defined in Section \ref{sec:abp estimates}. 

Note that  $B_{r}(x_0)\subset B_{1/32}$, $|\M \tilde{u}_{\tilde{a}}(\cdot,x_0)-\M \varphi| (\Omega\setminus  B_{1/32})=0$,
and 
\[
|\M \tilde{u}_{\tilde{a}}(\cdot,x_0)-\M \varphi|\bigl(B_{r}(x_0) \bigr) \le |\mu |\bigl(B_{r}(x_0)\bigr)+\omega_n \tilde{a}^n = |\mu |\bigl(B_{r}(x_0)\bigr)+|\mu|\bigl(\Omega\setminus B_{1/32}\bigr).
\]
The argument in Step 2.2 then applies and yields
\[
0\leq \varphi(x_0)-\tilde{u}_{\tilde{a}}(x_0 ,x_0) \leq  d_{n,0}b^2 + C(n)  \left(\frac{a^n}{r^{n-2}}+\frac{b^{2+\beta}}{r^{\beta}}\right).
\]

For the obstacle solution  $\tilde{v}_{\tilde{a}}(\cdot,\tilde{p})$, 
observe that $\tilde{v}_{\tilde{a}}(\cdot,\tilde{p})$ is linear on the $\tilde{K}_{\tilde{a}}$, and $x_0\in \tilde{K}_{\tilde{a}} \cap  B_{1/64}$.
Combining these observations with the estimates $\|\tilde{v}_{\tilde{a}}(\cdot,\tilde{p}) - \varphi\|_{L^{\infty}(\Omega)} \leq C(n) a |\Omega|^{\frac{1}{n}}$ and the regularity $\varphi \in C^{2,\alpha}(B_{1/32})$ (see \cite{caffarelli1990ilocalization,gutierrez2016monge}), we conclude that $\tilde{K}_{\tilde{a}} \subset B_{1/32}$ for sufficiently small $a$. This implies $|\M \tilde{v}_{\tilde{a}}(\cdot,\tilde{p})-\M \varphi| (\Omega\setminus  B_{1/32})=0$, and
\[
|\M \tilde{v}_{\tilde{a}}(\cdot,\tilde{p})-\M \varphi|\bigl(B_{r}(x_0) \bigr) \le |\mu |\bigl(B_{r}(x_0)\bigr)+\omega_n \tilde{a}^n =  |\mu |\bigl(B_{r}(x_0)\bigr)+|\mu|\bigl(\Omega\setminus B_{1/32}\bigr) .
\]
The argument in Step 2.3 then applies and yields
\[
0\leq \tilde{v}_{\tilde{a}}(x_0,\tilde{p}) - \varphi(x_0)\leq  d_{n,0}b^2 + C(n)  \left(\frac{a^n}{r^{n-2}}+\frac{b^{2+\beta}}{r^{\beta}}\right).
\] 

Combining the above estimates, we complete the proof of \eqref{eq:asymptotic-quadratic-order local}.

\medskip

\textbf{Step 4. Proof of \eqref{eq:asymptotic-quadratic-order local special} in the quadratic case.} 

By affine transformations and subtracting a linear function, we can assume, without loss of generality, that 
\[
\varphi(x)=\frac12|x|^{2}
\quad\text{and}\quad
x_{0}=0.
\]
The arguments in Steps~1--3 apply verbatim; in particular, \eqref{eq:theorem 2.3 lower general 1} and
\eqref{eq:theorem 2.3 upper general 1} remain valid. We will use the quadratic structure to strengthen the pointwise bounds
\eqref{eq:theorem 2.3 lower general} and \eqref{eq:theorem 2.3 upper general}.

\medskip

\textbf{Step 4.1. Improve the upper bound in Step~2.2.}

Fix $\rho:=r/2$. By the comparison principle,  
\[
\frac{1}{2}|x|^2 \geq \hat{u}(x)\geq W_{b}(x)-W_{b}\left(\rho  e_n\right)+\frac{1}{2}\rho^2-\| \hat{u} - \varphi \|_{L^{\infty}(\partial B_{\rho}(0))}\quad \text{in }B_{\rho}(0).
\]
Evaluating at $x=0$ and invoking \eqref{eq:theorem 2.3 lower general 1}, we obtain
\[
\begin{split}
\left|\hat{u}(0) -\varphi(0)\right| 
&\leq \| \hat{u} - \varphi \|_{L^{\infty}(\partial B_{\rho}(0))} + W_{b}\left(\rho  e_n\right)-\frac{1}{2}\rho^2 \\
&\leq \| \hat{u} - \varphi \|_{L^{\infty}(\partial B_{r/2}(0))} + d_{n,0} b^2 + C(n)  \frac{b^{n}}{\rho^{n-2}} \\
& \leq  d_{n,0}b^2 + C(n)  \frac{a^n}{r^{n-2}},
\end{split}
\]
where in the second inequality we used $W_{b}\left(\rho  e_n\right)=b^2W_1(\rho  e_n/b)$ and the expansion 
\[
W_1(x)=\frac{1}{2}|x|^2 +d_{n,0}   + O\left(|x|^{2-n}\right) \quad\mbox{as } x\to\infty.
\]

\medskip

\textbf{Step 4.2. Improve the upper bound in Step 2.3.}

Let $\rho:=r/2$ and define 
\[
w(x)=W_{b}^*(x-p)+p\cdot x-\frac{1}{2}|p|^2+d_{n,0}b^2+C(n)\frac{b^n}{\rho^{n-2}}, 
\]
where $p$ is the one in the definition of $\hat{v}_b(\cdot,p) $ in Step 2.3. Since $|p|\le Ca \ll \frac{r}{2}=\rho$, we have  
\[
\begin{split}
w(x)-\frac{1}{2}|x|^2 &=W_{b}^*(x-p)-\frac{1}{2}|x-p|^2+d_{n,0}b^2+C(n)\frac{b^n}{|x-p|^{n-2}} \\
&=W_{b}^*(x-p)-\frac{1}{2}|x-p|^2+d_{n,0}b^2+C(n)\frac{b^n}{\rho^{n-2}}  \geq 0 \quad \text{on } \partial   B_{\rho}(0).
\end{split}
\]
Moreover, $w$ and $\hat v$ correspond to obstacles with the same slope $p$. Hence, by the comparison principle in Lemma \ref{lem:comparison principle equal volume} to $w(x)-p\cdot x+\| \hat{v} - \varphi \|_{L^{\infty}(\partial B_{\rho}(0))}$ and $\hat{v}(x)-p\cdot x$, we have
\[
\frac{1}{2}|x|^2 \leq \hat{v}(x)\leq w(x)+\| \hat{v} - \varphi \|_{L^{\infty}(\partial B_{\rho}(0))} \quad \text{in }   B_{\rho}(0).
\]
Evaluating at $x=0$ and using \eqref{eq:theorem 2.3 lower general 1} yields 
\[
\begin{split}
\left|\hat{v}(0) -\varphi(0)\right| 
&\leq \| \hat{v} - \varphi \|_{L^{\infty}(\partial B_{\rho}(0))} + w(0) \\
& = \| \hat{v} - \varphi \|_{L^{\infty}(\partial B_{r/2}(0))} + W_{b}^*(-p)-\frac{1}{2}|p|^2+d_{n,0}b^2+C(n)\frac{b^n}{\rho^{n-2}} \\
& \leq \| \hat{v} - \varphi \|_{L^{\infty}(\partial B_{r/2}(0))} +d_{n,0}b^2+C(n)\frac{b^n}{\rho^{n-2}} \\
& \leq  d_{n,0}b^2 + C(n)  \frac{a^n}{r^{n-2}},
\end{split}
\]
where in the second inequality, we used $W_{b}^*\left(\rho  e_n\right)=b^2W_1^*(\rho  e_n/b)$ and the expansion 
\[
W_1^*(x)=\frac{1}{2}|x|^2 -d_{n,0} + O\left(|x|^{2-n}\right)\quad \mbox{as } x\to\infty.
\]

Combining the above estimates with the arguments in Steps 1--3, we conclude the proof of \eqref{eq:asymptotic-quadratic-order local special}.
\end{proof}

\section{Proof of the main results}\label{sec:main results proof}

We will use the following lemma showing that the section of an entire solution is bounded.
\begin{Lemma}\label{lem:volume upper bound}
Let $u$ be a convex function on a convex domain $\Omega$ satisfying $\det D^2 u = 1 + \mu$, where $\mu$ is a signed measure with $|\mu|(\Omega) = \omega_n a^n < \infty$ for some $a \ge 0$. 
Assume $0 \in \Omega$, $u(0) = \inf_\Omega u = 0$, and set $S_h := \{ x \in \overline{\Omega}: u(x) < h \}$. Then for every $h \ge a$, we have
\[
|S_h \cap \Omega| \leq C(n)h^\frac{n}{2}.
\]
\end{Lemma}
\begin{proof}
Without loss of generality, we assume $\Omega$ is bounded. For simplicity, we normalize $h=1$ via the rescaling $u(x) \mapsto u(h^{1/2}x)/h$ and assume $a \leq 1$. Let $\tilde{u}$ and $\varphi$ solves
\[
\det D^2 \tilde{u}=1+\mu \quad\text{in }  S_1, \quad  \tilde{u}=1  \quad \text{on } \partial S_1,
\]
and
\[
\det D^2 \varphi=1 \quad\text{in }  S_1, \quad  \varphi=1  \quad \text{on } \partial S_1,
\]
respectively. 
The comparison principle yields $\tilde{u} \ge u \ge 0$ in $S_1$. Moreover, by the Alexandrov estimate \eqref{eq:abp gene 1/n},
\[
\|\tilde{u}- \varphi\|_{L^{\infty}(S_1)} \leq C(n) a |S_1|^{\frac{1}{n}} \leq C(n) |S_1|^{\frac{1}{n}}.
\]
Since $S_1$ is a section of $\varphi$ and $\det D^2 \varphi = 1$, it follows from \cite{caffarelli1990ilocalization} that the height satisfies 
\[
1 - \inf_{S_1} \varphi\geq c(n)  |S_1|^{2/n}.
\] 
Combining these estimates, we conclude that $|S_1| \leq C(n)$ by noting that
\[
c(n)|S_1|^{2/n} \leq  (1 - \inf_{S_1} \varphi ) \leq  (1 - \inf_{S_1}\tilde{u})+C(n) |S_1|^{\frac{1}{n}} \leq 1+C(n) |S_1|^{\frac{1}{n}}.
\] 
\end{proof}

We now prove our theorems. 

\begin{proof}[Proof of Theorem \ref{thm:abp global existence}]
Without loss of generality, we assume $b =c= 0$.
Let us denote 
\[
E_R(x_0):=E_A(x_0, R) =\left\{x:\; (x-x_0)^{\top} A (x-x_0) <R^2\right\},\quad E_R:=E_{R}(0).
\]
For each $R >0$, let $u_R$ be the solution to 
\begin{equation}\label{eq:approx varphi w R}
\det D^2 u_R=1+\mu\cdot \chi_{E_{R/32}}  \quad\text{in }  E_R, \quad  u_R=\frac{1}{2} R^2\quad \text{on } \partial E_R.
\end{equation}
Applying \eqref{eq:asymptotic-quadratic-order local special} in Theorem~\ref{prop:decay estimate local} to
\[
w_R(x)=  R^{-2}u_R(RA^{-\frac{1}{2}}x) \quad \text{and} \quad \varphi(x)=\frac{1}{2} |x|^2
\]
with $\Omega=B_{1}(0)$, we obtain that for every $0<r\le 1/64$ the following estimate holds: 
\[
\left|w_R(x) -\frac{1}{2} |x|^2\right| \leq d_{n,0}b_{r,R}^2 + C(n)\frac{a_R^n}{r^{n-2}} \quad \text{in } B_{1/64}(0),
\]
where $a_R \geq  b_{r,R}\geq 0$ satisfy
\[
\omega_n a_R^n=|\mu\cdot \chi_{E_{R/32}}|(E_R(0))/R^n, \quad 
\omega_n b_{r,R}^n=|\mu\cdot \chi_{E_{R/32}}|\bigl(E_{rR}(R A^{-\frac{1}{2}} x)\bigr) /R^n.
\] 
For every $\rho>0$, we choose $r=\rho/R\le 1/64$ if $R$ is large. Scaling back and noting that $|\mu\cdot \chi_{E_{R/32}}|(E_{R}(0)) \leq |\mu|(\Omega) = \omega_n a^n$, we find that
\begin{equation}\label{eq:decay R rho}
\left| u_R(x) - \frac{1}{2}x^{\top}Ax   \right| \leq d_{n,0}b_{\rho}^2+ C(n)  \frac{a^n}{\rho^{n-2} } 
\quad \text{in } E_{R/64}(0),
\end{equation}
and $0 \leq b_{\rho} \leq a$ is given by 
\[
\omega_n b_{\rho}^n:=|\mu\cdot \chi_{E_{R/32}}|(E_{\rho}(x)) \leq \omega_na^n.
\]
By \eqref{eq:decay R rho}, the family ${u_R}$ is locally uniformly bounded  in $R$ for large $R$ and fixed $x$. Hence, a subsequence of ${u_R}$ converges locally uniformly to an entire solution $u$ of \eqref{eq:global-ma-eq-mu}.

Taking the limit along a subsequence of $R \to \infty$ in \eqref{eq:decay R rho}, we obtain
\[
\left| u(x) - \frac{1}{2}x^{\top}Ax   \right| \le  d_{n,0}\bigl(\omega_n^{-1}|\mu|(E_A(x,\rho))\bigr)^{\frac{2}{n}} + C(n)  \frac{a^n}{\rho^{n-2}} .
\]  
This establishes \eqref{eq:global decay}.
Setting $\rho = |x|/2$ in \eqref{eq:global decay} yields \eqref{eq:asymptotic-quadratic}. Taking the limit as $\rho \to \infty$ implies \eqref{eq:abp-global}.

Uniqueness follows from the comparison principle. 

This establishes Theorem \ref{thm:abp global existence}. 

Furthermore, this unique solution satisfies \eqref{eq:abp-global} and  \eqref{eq:global decay}.
\end{proof}

\begin{proof}[Proof of Theorem \ref{thm:abp global rigidity}]  
We consider $u$ is a solution of \eqref{eq:global-ma-eq-mu} with $\mu=\M u-\LL$. 

\medskip

\textbf{Step 1.} We will first prove part (i).

Without loss of generality, we assume $u(0) = \inf_{\R^n} u = 0$.
For each large $R >0$, let us denote $D_R:=\{x:\; u(x)< \frac{1}{2}R^2 \}$.  Lemma \ref{lem:volume upper bound} gives the volume bound $|D_R| \le C(n) R^n$. Since $u$ is locally Lipschitz and $D_1 \subset D_R$, the set $D_R$ is bounded.
Let $\varphi_R$ solve
\begin{equation*}\label{eq:approx u varphi R}
\det D^2 \varphi_R=1 \quad\text{in }   D_R, \quad  \varphi_R= u = \frac{1}{2}R^2\quad \text{on } \partial D_R.
\end{equation*}
Observe that $D_R$ is a section of $\varphi_R$. We can denote $S_{h_R}(x_R)=S_{h_R}^{\varphi_R}(x_R)=D_R$ for some $x_R \in D_R$ satisfying $\varphi_R(x_R)=\inf_{D_R} \varphi_R= \frac{1}{2}R^2-h_R$.
The Alexandrov estimate \eqref{eq:abp gene 1/n} implies 
\[
\|u- \varphi_R\|_{L^{\infty}(D_R)}  \leq  C(n)ah_R^{\frac{1}{2}}.
\]
Evaluation at $0$ and $x_R$ yields $|h_R - \frac{1}{2}R^2| \leq C(n) a h_R^{1/2}$. This implies $h_R \leq \frac{1}{2}R^2 + C(n)a^2$. 
For $R \geq C(n)a$, we further obtain 
\[
\left|h_R - \frac{1}{2}R^2\right| \leq C(n) a R\quad  \text{and } \quad  \|u - {\varphi}_R\|_{L^{\infty}(D_R)}  \leq  C(n)aR.
\] 

Since $\det D^2 \varphi_R=1$,  there exists an affine transformation $A_R \in \mathcal{A}_n$ such that
\[
A_Rh_R^{\frac{1}{2}} B_{1/C_0} \subset (S_{h_R}(x_R) -x_R) \subset A_R h_R^{\frac{1}{2}}B_{C_0} 
\]
for some positive constant $C_0(n)$.
Let us consider the transforms $T_Rx= A_RRx+x_R$ and the normalized functions
\[
\tilde{u}_R(x)=  \frac{ u(T_R x)}{R^2} -\frac{1}{2} ,\quad \tilde{\varphi}_R(x) =  \frac{\varphi_R(T_R x)}{R^2} -\frac{1}{2}  ,\quad \text{in } \widetilde{\Omega}_R =T_R^{-1}D_R.
\] 
Then, $B_{1/(2C_0)}  \subset \widetilde{\Omega}_R \subset B_{2C_0}$.
Applying \eqref{eq:asymptotic-quadratic-order local} in Theorem~\ref{prop:decay estimate local} to $\frac{\tilde{u}_R(2C_0x)}{4C_0^2}$ and $\frac{\tilde{\varphi}_R(2C_0x)}{4C_0^2}$ in $2\widetilde{\Omega}_R$ yields for every $0< r \leq 1/(128C_0)$:
\begin{equation}\label{eq:decay R rho 2normal}
\left|\tilde{u}_R(x) -\tilde{\varphi}_R(x)\right| \leq d_{n,0}\tilde{b}_{r,R}^2 + C(n)  \left(\frac{\tilde{a}_R^n}{r^{n-2}}+
\frac{\tilde{b}_{r,R}^{2+\beta}}{r^{\beta}}
\right) \quad \text{in } B_{1/(128C_0)}(0),
\end{equation}
where $\beta = \frac{(n-2)}{2n+1}$, and $\tilde{a}_R\geq  \tilde{b}_{r,R}\geq 0$ satisfy
\[
\omega_n \tilde{a}_R^n=R^{-n}|\mu|(D_R) \leq \omega_n R^{-n}a^n, \quad \omega_n \tilde{b}_{r,R}^n=R^{-n}(|\mu|\bigl(T_RB_r(x)\bigr)  +|\mu|\bigl(D_R \setminus T_R B_{1/32}\bigr) ) .
\]   
Since $\tilde{\varphi}_R$ is locally $C^{2,\alpha}$ and uniformly convex in $\widetilde{\Omega}_R$, for any  $x \in B_{1/(64C_0)}(0)$, we have $S_{c(n)r^2}^{\tilde{\varphi}_R}(x) \subset B_{r}(x) \subset S_{C(n)r^2}^{\tilde{\varphi}_R}(x)  $.
From this inclusion we obtain 
\[
\omega_n \tilde{b}_{r,R}^n  \leq R^{-n}\left(|\mu|\bigl(S_{Cr^2R^2}^{\varphi_R}(T_R x)\bigr)+|\mu|\bigl(D_R \setminus  T_R B_{1/32}\bigr)\right) .
\]
Note that the function $\tilde{\varphi}_R \in \C_{+}^{2,\alpha}(B_{1/(4C_0)})$ attains its minimum at $0$, and $\tilde{u}_R$ attains its minimum at $T_R^{-1}(0)$. Then from the uniform bound 
\[
\|\tilde{u}_R - \tilde{\varphi}_R\|_{L^{\infty}(\widetilde{\Omega}_R)} =R^{-2} \|u - {\varphi}_R\|_{L^{\infty}(D_R)} \leq  C(n)aR^{-1},
\]
we obtain
\[
\begin{split}
\tilde{\varphi}_R(T_R^{-1}(0))
&\geq \tilde{\varphi}_R(0) +c|T_R^{-1}(0)|^2 \\
&\geq \tilde{u}_R (0) -C(n)aR^{-1} +c|T_R^{-1}(0)|^2 \\
&\geq \tilde{u}_R(T_R^{-1}(0)) -C(n)aR^{-1} +c|T_R^{-1}(0)|^2 \\
&\geq \tilde{\varphi}_R(T_R^{-1}(0)) -2C(n)aR^{-1} +c|T_R^{-1}(0)|^2 .
\end{split}
\]
This implies $|T_R^{-1}(0)|$ is small when $R$ is large; hence $ c_1 D_R \subset  T_R B_{1/(128C_0)}(0)$ holds for small $c_1(n)> 0$. 
Scaling back the estimate \eqref{eq:decay R rho 2normal}, we now derive
\begin{equation}\label{eq:decay R rho 2}
\left| u(x)-\varphi_R(x) \right| \leq d_{n,0}b_{\rho}^2+ C(n)  \left( \frac{a^n}{\rho^{n-2} }+
\frac{b_{\rho}^{2+\beta}}{\rho^{\beta}}
\right) \quad \text{in } c_1(n)D_R,
\end{equation}
where $\rho=rR$ and $b_{\rho} \geq 0$ satisfies
\[
\omega_n b_{\rho}^n\leq |\mu|(S_{C(n)\rho^2}^{\varphi_R}(x)) +|\mu|\bigl(D_R \setminus  T_R B_{1/32}\bigr) \leq |\mu|(S_{C(n)\rho^2}^{\varphi_R}(x)) +|\mu|\bigl(D_R \setminus  c_1D_R\bigr).
\]
By \eqref{eq:decay R rho 2}, the family ${\varphi_R}$ is locally  uniformly bounded for large $R$. Thus, a subsequence converges locally uniformly to an entire solution of $\det D^2 \varphi = 1$. The Jörgens-Calabi-Pogorelov theorem then implies that $\varphi(x) = \frac{1}{2}x^{\top}Ax + b \cdot x + c$ for some $A \in \A_n$, $b \in \mathbb{R}^n$, and $c \in \mathbb{R}$. Taking $R \to \infty$ in \eqref{eq:decay R rho 2}, we find that
\[
\begin{split}
& \left| u(x) -\left(\frac{1}{2}x^{\top}Ax + b \cdot x + c\right)   \right|  \\
& \le  d_{n,0}\bigl(\omega_n^{-1}|\mu|(E_A(x,C(n)\rho))\bigr)^{\frac{2}{n}} +  C(n)\left(\frac{a^n}{\rho^{n-2}}+
\frac{\bigl(\omega_n^{-1}|\mu|(E_A(x,C(n)\rho))\bigr)^{\frac{2+\beta}{n}}  }{\rho^{\beta}}
\right).
\end{split}
\] 
Setting $2C(n)\rho = |x| $ yields \eqref{eq:asymptotic-quadratic}.

\medskip

\textbf{Step 2.}  By Theorem \ref{thm:abp global existence}, this function $u$ is the unique solution of \eqref{eq:global-ma-eq-mu} with $\mu=\M u-\LL$ satisfying \eqref{eq:asymptotic-quadratic}. As mentioned in the end of the proof of Theorem \ref{thm:abp global existence}, $u$ satisfies \eqref{eq:abp-global}. This proves part (ii) of Theorem \ref{thm:abp global rigidity}.

\medskip

\textbf{Step 3.}  We now prove part (iii) in Theorem \ref{thm:abp global rigidity}, that is, the equality
\[
\left\| u(x) - \left(\frac{1}{2}x^{\top}Ax  + b \cdot x + c\right) \right\|_{L^{\infty}(\R^n)} =d_{n,0} a^2
\]
holds if and only if, up to a unimodular affine transformation and the addition of linear functions, $u$ equals the function $W_a$ defined in \eqref{eq:isolated global solution} or $W_a^*$ defined in \eqref{eq:obstacle global solution}.

By the asymptotic expansion \eqref{eq:asymptotic-quadratic}, we may assume that the supremum norm is attained at the origin. 
We also assume $b =c= 0$, and thus $u(0) = \pm d_{n,0}a^2$.
Consider the rescaled solutions $u_R$ defined in \eqref{eq:approx varphi w R}. 
For a given small $\sigma > 0$, apply Theorem \ref{thm:rigidity} to $w_R(x) = R^{-2} u_R(R A^{1/2} x)$ and $\varphi(x) = \frac{1}{2} |x|^2$ on the domain $\Omega = B_1(0)$ and then passing to the limit $R \to \infty$ yields
\begin{align*}
\frac{ u(0)-\varphi(0)}{d_{n,0}a^2} 
& \geq -1+C(n,\sigma)\left(\frac{\omega_na^n-\mu \left( E_{\sigma a}(0)  \right) }{\omega_na^n} \right)^\frac{n}{2}, \\
\frac{ u(0)-\varphi(0)}{d_{n,0}a^2} 
& \leq 1-C(n,\sigma)\left(\frac{\omega_na^n+\mu \left(E_{(1+\sigma)a} (0)  \right)}{ \omega_na^n} \right)^n,
\end{align*}
where $E_R(0)=E_A(0, R) $.

Note that $\varphi(0)=0$. If $u(0) = -d_{n,0}a^2$, then we have $\mu(E_{\sigma a}(0)) = \omega_n a^n > 0$, and the measure $\mu$ must be a non-negative measure supported in $E_{\sigma a}(0)$. If $u(0) = d_{n,0}a^2$, then $\mu(E_{(1+\sigma)a}(0)) = -\omega_n a^n < 0$, and the measure $\mu$ must be a non-positive measure supported in $E_{(1+\sigma)a}(0)$.
Since $\sigma > 0$ can be arbitrary small and $1 + \mu \geq 0$, we conclude that either $\mu = \omega_n a^n \delta_0$ or $\mu = -\chi_{E_{a}(0)}$. Then it follows from \cite{jin2016solutions} that, up to a unimodular affine transformation and the addition of linear functions, $u$ equals the function $W_a$ or $W_a^*$. 
\end{proof}

 \begin{proof}[Proof of Theorem \ref{thm:global decay}]
 By Theorem \ref{thm:abp global existence}, this function $u$ is the unique solution of \eqref{eq:global-ma-eq-mu} with $\mu=\M u-\LL$ satisfying \eqref{eq:asymptotic-quadratic}. As mentioned in the end of the proof of Theorem \ref{thm:abp global existence}, $u$ satisfies \eqref{eq:global decay} . This proves Theorem \ref{thm:global decay}.
 \end{proof}

\begin{proof}[Proof of Theorem \ref{prop:strictly convex cond global}] 
After an appropriate affine transformation, we may assume that $u$ asymptotically approaches $ \frac{1}{2}|x|^2$ at infinity in the sense of \eqref{eq:asymptotic-quadratic}.
Denoting $a^n = \sum_{i=1}^{\infty} a_i^n$, the estimate \eqref{eq:abp-global} implies
\[
 \frac{1}{2}|x|^2 -d_{n,0}a^2 \leq u(x) \leq \frac{1}{2}|x|^2.
\]  

Now, suppose for contradiction that $u$ is not strictly convex. By the results in \cite{caffarelli1990ilocalization, caffarelli1993note}, it must then be linear on a non-degenerate $E$ with vertices in $\{y_1, y_2, \dots\}$. After a unimodular affine transformation and the addition of linear functions, we may assume that $u$ is linear on the segment connecting $y_1 = b e_n$ and $y_2 =-be_n$ for some $b>0$. It follows that
\[
0\geq 2u(0)= u(y_1)+u(y_2) \geq 2\left(\frac{b^2}{2}-d_{n,0}a^2\right)= b^2-2d_{n,0}a^2,
\]
which implies $b^2 \leq 2d_{n,0}a^2$. Consequently,
\[
|y_1-y_2|^n =(2b)^n\leq(8d_{n,0})^{\frac{n}{2}}\sum_{i=1}^{\infty}  a_i^n.
\]
This contradicts \eqref{eq:sufficient cond for sc}; hence, $u$ must be strictly convex. The regularity theory established in \cite{caffarelli1990ilocalization,caffarelli1990interiorw2p,caffarelli1991regularity} and \cite{savin2005obstacle,huang2024regularity} then becomes applicable, yielding the desired regularity for $u$.
\end{proof}

\begin{proof}[Proof of Theorem \ref{thm:global deviation}.]
Let $\mu:=\M u-\LL.$ If $|\mu|(\mathbb{R}^n) = 0$, then it follows from the Jörgens-Calabi-Pogorelov theorem. If $|\mu|(\mathbb{R}^n) = \infty$, the estimate \eqref{eq:global deviation} is immediate since the right-hand side is infinite.

We therefore assume $0<|\mu|(\mathbb{R}^n) < \infty$. By Theorem \ref{thm:abp global rigidity}, after a suitable unimodular affine transformation we may assume for simplicity that
\[
\limsup_{|x| \to \infty} \left| u(x) -  \frac{1}{2}|x|^2   \right| = 0.
\]
Let us denote
\[
\mathcal{P}=\left\{P(x)=\frac{1}{2}x^{\top}Ax  + b \cdot x+c:\; A \in \A_n,\ b \in \mathbb{R}^n,\ c \in \mathbb{R} \right\}.
\]
Then any polynomial $P\in\mathcal{P}$ satisfying $\left\| u(x) - P(x) \right\|_{L^{\infty}(\R^n)}<\infty$ must be of the form $P(x)=\frac{1}{2}|x|^{2}+c$ for some $c\in\R$. Consequently,
\begin{equation}\label{eq:global deviation reduction}
\begin{split}
\inf_{P\in\mathcal{P}}\bigl\|u(x)-P(x)\bigr\|_{L^{\infty}(\mathbb{R}^{n})}  &=\inf_{c\in\mathbb{R}}\left\|u(x)-\frac{1}{2}|x|^{2}-c\right\|_{L^{\infty}(\R^{n})}  \\
&=\frac{1}{2}\left(\sup_{\R^n}\left( u(x)-\frac{1}{2}|x|^2\right)-\inf_{\R^n} \left(u(x)-\frac{1}{2}|x|^2\right)\right).
\end{split}
\end{equation} 
Therefore, it suffices to estimate for
\[
\inf_{\R^n} \left(u(x)-\frac{1}{2}|x|^2\right)\quad \text{and} \quad \sup_{\R^n}\left( u(x)-\frac{1}{2}|x|^2\right) .
\] 

\medskip

\textbf{Step 1. Verification of  \eqref{eq:global deviation}.} Assume $0<|\mu|(\mathbb{R}^n) < \infty$. Let us express
\[ 
\mu=\mu_+-\mu_-,  
\]
where $\mu_+$ and $\mu_-$ are non-negative measures satisfying $\mu_+ \perp \mu_-$, and denote
\[
|\mu|(\R^n)=\omega_n a^n,\quad \mu_+(\R^n) =\omega_n a_+^n,\quad \mu_-(\R^n) =\omega_n a_-^n
\]
for  non-negative $a_+, a_-, a$. Recalling Theorem \ref{thm:abp global rigidity},  
let $u_+$ and $u_-$ be the convex solutions to  
\begin{align*}
& \M u_{+}=1+\mu_+  \quad \text{in } {\R^n}, \quad \limsup_{|x| \to \infty} \left| u_{+}(x) -  \frac{1}{2}|x|^2   \right| = 0, \\
& \M   u_{-} =1-\mu_- \quad \text{in } {\R^n}, \quad   \limsup_{|x| \to \infty} \left| u_{-}(x) -  \frac{1}{2}|x|^2   \right| = 0,  
\end{align*} 
respectively. 
Then, by the comparison principle and Theorem \ref{thm:abp global rigidity}, we have
\[
\frac{1}{2}|x|^2 -d_{n,0} a_{+}^2 \leq u_{+} \leq u \leq u_{-} \leq \frac{1}{2}|x|^2 +d_{n,0} a_-^2.
\]
Consequently,
\begin{equation}\label{eq:global deviation mid}
-d_{n,0} a_{+}^2 \leq u_{+}(x) -\frac{1}{2}|x|^2\leq u(x)- \frac{1}{2}|x|^2\leq u_{-}(x)-\frac{1}{2}|x|^2 \leq   d_{n,0} a_-^2.
\end{equation}
Recalling \eqref{eq:global deviation reduction}, we find that  
\begin{equation}\label{eq:global deviation mid2}
\inf_{P\in\mathcal{P}}\bigl\|u(x)-P(x)\bigr\|_{L^{\infty}(\mathbb{R}^{n})} \leq \frac{1}{2}d_{n,0}( a_-^2+a_+^2)\leq 2^{-\frac{2}{n}} d_{n,0}( a_-^n+a_+^n)^{\frac{2}{n}} = 2^{-\frac{2}{n}} d_{n,0}a^2.
\end{equation}
This establishes \eqref{eq:global deviation}.

\medskip

\textbf{Step 2. Strict inequality in  \eqref{eq:global deviation}.} 
We are going to prove that equality in \eqref{eq:global deviation} can never be attained when $0<|\mu|(\mathbb{R}^n) < \infty$. 

Suppose, for contradiction, that equality in \eqref{eq:global deviation} holds for some convex function $u$ with $|\mu|(\mathbb{R}^n)\in (0,\infty)$. 
Let $a_+$, $a_-$, $u_+$ and $u_-$ be as in Step 1. 
Then, by \eqref{eq:global deviation reduction} and \eqref{eq:global deviation mid2}, we have $a_+=a_->0$ and
\begin{align*}
\inf_{\R^n} \left(u(x) -\frac{1}{2}|x|^2\right)=-d_{n,0} a_{+}^2 
\quad \text{and} \quad 
\sup_{\R^n}\left( u(x)-\frac{1}{2}|x|^2\right) =   d_{n,0} a_-^2 .
\end{align*}
Moreover, recalling the asymptotic behavior of $u$, we may assume that these extremal values are attained at $y_{+}$ and $y_{-}$, respectively. 
It then follows from \eqref{eq:global deviation mid} that
\begin{align*}
\inf_{\R^n} \left(u_{+}(x) -\frac{1}{2}|x|^2\right)=-d_{n,0} a_{+}^2
\quad \text{and} \quad
\sup_{\R^n}\left( u_{-}(x)-\frac{1}{2}|x|^2\right) =   d_{n,0} a_-^2,
\end{align*}
and these extremal values are attained at $y_{+}$ and $y_{-}$, respectively.
Consequently,
\[
\left\| u_{+}(x) -\frac{1}{2}|x|^2 \right\|_{L^{\infty}(\R^n)} = d_{n,0} a_{+}^2 \quad \text{and} \quad \left\| u_{-}(x) -\frac{1}{2}|x|^2 \right\|_{L^{\infty}(\R^n)} = d_{n,0} a_{-}^2.
\]
Combining this with Theorem~\ref{thm:abp global rigidity} and the prescribed asymptotic behavior of $u_{+}$ and $u_{-}$, we conclude that, modulo the addition of (possibly different) linear functions, $u_{+}$ coincides with the function $W_{a_{+}}$ defined in~\eqref{eq:isolated global solution}, while $u_{-}$ coincides with the function $W_{a_{-}}^{*}$ defined in~\eqref{eq:obstacle global solution}. 

In summary, we have $u_+ < \frac{1}{2}|x|^2 < u_-$. 
Moreover, $\operatorname{supp}(\mu)$ is compact. Consequently, since $\{y_{+},y_{-}\}$ is finite, there exists $M>0$ such that 
\[
\{y_+,y_{-}\} \cup \operatorname{supp}(\mu) \subset B_M
\]
Since $u_{+}\leq u\leq u_{-}$ on $\R^{n}$, the strong maximum principle implies that either  $u < u_-$ on $\R^n \setminus B_M$ or $u_+ < u$ on $\R^n \setminus B_M$. 
If $u < u_-$ on $\R^n \setminus B_M$, then the comparison principle yields  
\[
u(x) \leq u_-(x)-\inf_{\partial B_{2M}} (u_--u) \leq \frac{1}{2}|x|^2  +d_{n,0} a_-^2-\inf_{\partial B_{2M}} (u_--u) \quad \text{in } B_{2M}.
\]
Since $y_+ \in B_{2M}$, this contradicts $u(y_-) -\frac{1}{2}|y_-|^2 =d_{n,0} a_{-}^2 $.
If $u_{+}<u$ on $\mathbb{R}^{n}\setminus B_{M}$, 
then the comparison principle yields  
\[
u(x) \geq u_+(x)+\inf_{\partial B_{2M}} (u-u_{+}) \geq \frac{1}{2}|x|^2  -d_{n,0} a_+^2+\inf_{\partial B_{2M}} (u-u_{+}) \quad \text{in } B_{2M},
\]
and we obtain a contradiction with $u(y_+) -\frac{1}{2}|y_+|^2 =-d_{n,0} a_{+}^2 $.
In conclusion, equality in~\eqref{eq:global deviation} cannot be attained.

\medskip

\textbf{Step 3. Sharpness of the constant.} 
Assuming $0<|\mu|(\mathbb{R}^n) < \infty$, we are going to prove that the constant $2^{-\frac{2}{n}} d_{n,0}$ is optimal via a family of solutions consisting of an isolated singularity and an obstacle, in which the distance between the singularity and the coincidence set tends to infinity.

For $\rho>0$, $M\gg \rho^2$, we consider the obstacle problem
\[
\det D^2 u_{M,\rho}=\left(1+\frac{1}{2}\omega_n a^n\delta_0\right)\cdot \chi_{\{ u_{M,\rho}> \rho^2x_n+h_{M,\rho} \}} \quad \text{in } B_{M}, \text{ and }  u_{M,\rho} =\frac{1}{2}|x|^2 \quad \text{on } \partial B_{M},
\]
with $u_{M,\rho}\geq \rho^2x_n+h_{M,\rho}$, where $h_{M,\rho}$ is chosen such that coincidence set 
\[
K_{M,\rho}:=\{ u_{M,\rho}= \rho^2x_n+h_{M,\rho} \}
\] 
satisfies $|K_{M,\rho}|=\frac{1}{2}\omega_n a^n$. 
For $\rho$ and $M$ sufficiently large, the estimate
\[
\left\|u_{M,\rho} - \frac{1}{2}|x|^2\right\|_{L^\infty(B_M)} \le C(n) a^2
\] 
(which follows from Theorem \ref{thm:ordera2}) yields 
\[
\frac{1}{2}|x|^2 -C(n) a^2\leq \rho^2x_n+h_{M,\rho} \leq \frac{1}{2}|x|^2 +C(n) a^2 \quad \text{on } K_{M,\rho},
\]
and a further elementary analysis gives
\[
\left|h_{M,\rho}+\frac{1}{2} \rho^4\right| \leq Ca^2,\quad  K_{M,\rho}\subset B_{Ca}(\rho^2e_n) .
\] 
Passing to the limit  $M \to \infty$ along a subsequence, we obtain an entire convex solution of 
\[
\det D^2 u_\rho= 1+\frac{1}{2}\omega_n a^n\delta_0  -  \chi_{\{ u_{\rho}= \rho^2x_n+h_\rho\}} =:1+\mu_\rho  \quad \text{in } \R^n, \quad u_\rho\geq \rho^2x_n+h_\rho,
\] 
whose coincidence set 
\[
K_{\rho}:=\{ u_\rho= \rho^2x_n+h_\rho \}\subset B_{Ca}(\rho^2e_n)
\]
satisfying $|K_{\rho}|=\frac{1}{2}\omega_n a^n$. 
Using Theorem \ref{thm:abp global rigidity} and adding some constant if needed, we can assume that
\begin{equation*}\label{eq:urhoinfinity}
\limsup_{|x| \to \infty} \left| u_\rho(x) -  \frac{1}{2}|x|^2   \right| = 0.
\end{equation*}

From \eqref{eq:global decay} and since $|\mu_\rho|(B_{\rho/2} (x))=0$ for $x\in \partial B_{\rho}$, we have
\[
\left| u_\rho(x) -  \frac{1}{2}|x|^2   \right|
 \le   C(n)  \frac{a^n}{\rho^{n-2}}, \quad \forall x\in \partial B_{\rho}.
\]
Denote $\tilde{a}=2^{-\frac{1}{n}}a$. The comparison principle then implies 
\[
u_\rho(x) \leq W_{\tilde{a}}(x)-W_{\tilde{a}}(\rho e_n)+\frac{1}{2}\rho^2+C(n)  \frac{a^n}{\rho^{n-2}}\quad \text{in } B_{\rho},
\]
which gives
\begin{equation}\label{eq:lower deviation}
u_{\rho}(0)\leq -d_{n,0} \tilde{a}^2+C(n)  \frac{a^n}{\rho^{n-2}}.
\end{equation}

Similarly, we have
\[
\left| u_\rho(x) -  \frac{1}{2}|x|^2   \right|
 \le   C(n)  \frac{a^n}{\rho^{n-2}}, \quad \forall x\in \partial B_{\rho}(\rho^2e_n).
\]
The comparison principle (see Lemma \ref{lem:comparison principle equal volume}) then implies 
\[
u_\rho(x)-\rho^2 x_n+\frac{1}{2}\rho^4 \geq W_{\tilde{a}}^*(x-\rho^2e_n)-W_{\tilde{a}}^*(\rho e_n)+\frac{1}{2}\rho^2-C(n)  \frac{a^n}{\rho^{n-2}}\quad \text{in } B_{\rho}(\rho^2e_n),
\]
which gives
\begin{equation}\label{eq:upper deviation}
u_{\rho}(\rho^2e_n)-\frac{1}{2}\rho^4\geq d_{n,0} \tilde{a}^2-C(n)  \frac{a^n}{\rho^{n-2}}.
\end{equation}
Combining~\eqref{eq:lower deviation} and~\eqref{eq:upper deviation} with~\eqref{eq:global deviation reduction}, we obtain
\[
\inf_{P\in \mathcal{P}}\left\| u_{\rho}(x) - P(x) \right\|_{L^{\infty}(\R^n)}=\inf_{c\in R}\left\| u_{\rho}(x) - \frac{1}{2}|x|^2-c \right\|_{L^{\infty}(\R^n)} \geq d_{n,0} \tilde{a}^2-C(n)  \frac{a^n}{\rho^{n-2}}.
\]
Letting $\rho\to\infty$, we recover the sharp constant  $d_{n,0} \tilde{a}^2=2^{-\frac{2}{n}}d_{n,0}a^2$.
\end{proof}

\appendix
\counterwithin*{equation}{section}
\renewcommand\theequation{\thesection\arabic{equation}}

\small

\bibliography{references}

\smallskip

\noindent T. Jin

\noindent Department of Mathematics, The Hong Kong University of Science and Technology\\
Clear Water Bay, Kowloon, Hong Kong\\
Email: \textsf{tianlingjin@ust.hk}

\medskip

\noindent X. Tu

\noindent Department of Mathematics, The Hong Kong University of Science and Technology\\
Clear Water Bay, Kowloon, Hong Kong\\[1mm]
Email:  \textsf{maxstu@ust.hk}

\medskip

\noindent J. Xiong

\noindent School of Mathematical Sciences, Laboratory of Mathematics and Complex Systems, MOE\\ Beijing Normal University, 
Beijing 100875, China\\
Email: \textsf{jx@bnu.edu.cn}

\end{document}